\documentclass{article}

\usepackage{amsmath,amsthm,amssymb,latexsym,mathabx}
\usepackage{tikz}
\usepackage{mathrsfs,upgreek}
\usepackage{enumerate}

\newtheorem{theorem}{Theorem}[section]
\newtheorem{lemma}[theorem]{Lemma}
\newtheorem{proposition}[theorem]{Proposition}
\newtheorem{corollary}[theorem]{Corollary}

\theoremstyle{definition}

\newtheorem{example}[theorem]{Example}
\newtheorem{remark}[theorem]{Remark}

\newcommand{\op}[1]{\textrm{\upshape #1}}

\newcommand{\join}{\vee}
\renewcommand{\Join}{\bigvee}
\newcommand{\meet}{\wedge}

\newcommand{\la}{\langle}
\newcommand{\ra}{\rangle}
\newcommand{\alg}[1]{{\textbf{\upshape #1}}}  %
\newcommand{\vv}[1]{\mathsf {#1}}

\renewcommand{\a}{\alpha}
\renewcommand{\b}{\beta}
\renewcommand{\d}{\delta}
\newcommand{\f}{\varphi}
\newcommand{\g}{\gamma}
\newcommand{\e}{\varepsilon}
\renewcommand{\th}{\theta}
\renewcommand{\o}{\omega}

\newcommand{\cg}{\vartheta}
\newcommand{\Th}{\varTheta}
\newcommand{\sse}{\subseteq}
\newcommand{\ap}{\approx}
\newcommand{\app}{\approx}
\newcommand{\HH}{{\mathbf H}}  
\newcommand{\II}{{\mathbf I}} 
\newcommand{\SU}{{\mathbf S}} 
\newcommand{\PP}{{\mathbf P}}   
\newcommand{\VV}{{\mathbf V}}   

\newcommand{\QQ}{{\mathbf Q}}

\newcommand{\ib}{\item[$\bullet$]}

\newcommand{\Con}[1]{\operatorname{Con}(\alg #1)}
\newcommand{\vucc}[2]{#1_0,\dots,#1_{#2}}
\newcommand{\vuc}[2]{#1_1,\dots,#1_{#2}}

\newcommand{\imp}{\rightarrow}

\newcommand{\ol}[1]{\overline #1}

\makeatletter
\def\square{\RIfM@\bgroup\else$\bgroup\aftergroup$\fi
  \vcenter{\hrule\hbox{\vrule\@height.6em\kern.6em\vrule}\hrule}\egroup}
\makeatother
\newcommand{\smlcirc}{\raise3pt\hbox{\textrm{\circle{3.3}}}}

\newcommand{\myfrac}[2]{\dfrac{#1}{\lower.5ex\hbox{$#2$}}}
\mathchardef\hu="0362

\renewcommand{\e}{\varepsilon}

\newcommand{\Def}[1]{\textbf{#1}}

\begin{document}

\title{Structural completeness in quasivarieties}
\author{Paolo Aglian\`{o}\\
DIISM,\\
Universit\`a di Siena, Italy\\
agliano@live.com
\and
 Alex Citkin\\
 Metropolitan Telecommunications,\\
New York, USA\\
 acitkin@gmail.com}
\date{}
\maketitle

\begin{abstract} In this paper we study various forms of (hereditary) structural completeness for quasivarieties of algebras, using mostly algebraic techniques. More specifically
we study relative weakly projective algebras and the way they interact with structural completeness in quasivarieties. These ideas are then applied to the study of $C$-structural
completeness and $C$-primitivity, through an algebraic generalization of Prucnal's substitution.  Finally we study in depth dual i-discriminator quasivarieties in which a particular instance
of Prucnal's substitution is used to prove that if  each fundamental operation commutes with the i-discriminator, then it is primitive.
\end{abstract}

\section{Introduction.}

Research in admissibility of (finitary structural inference) rules in propositional logics is concerned with primarily two problems: given a propositional logic $\vv L$, does $\vv L$ has admissible not derivable rules? And if it does, what are they (see, e.g.,\cite{Rybakov1997})? The logics in which all structural admissible rules are derivable are called \Def{structurally complete}, and this notion was introduced in \cite{Pogorzelski1971}. Since a structurally complete logic cannot have  structurally complete extensions (e.g.  Medvedev Logic (see \cite{Citkin1977}), it is natural to ask whether a logic and  all its extensions are  structurally complete;  for such logics the term \Def{hereditary structural completeness} was coined in \cite{Citkin1977}.
Thus, traditionally, the problem of structural completeness is formulated in the following way: given a propositional logic $\vv L$, is it structurally complete? And if it is, is it hereditarily structurally complete? And if it is not, which of its extensions  are structurally complete?

The equivalent algebraic semantic of an algebraizable propositional logic is a quasivariety (see \cite{BlokPigozzi1989}): the formulas are translated into equations and the rules  into quasiequations. In algebraic terms, given a quasivariety $\vv Q$, a quasiequation $\Phi$ is admissible in $\vv Q$ if $\vv F_\vv Q(\omega) \models \Phi$, and $\Phi$ is derivable in $\vv Q$ if $\vv Q \models \Phi$. Accordingly, $\vv Q$ is \Def{structurally complete} if $\vv Q$ does not have admissible not derivable quasiequations, that is, if $\vv Q$ is generated by $\vv F_\vv Q(\omega) $. And $\vv Q$ is \Def{hereditarily structurally complete} if $\vv Q$ and all its subquasivarieties are structurally complete. Because hereditary structural completeness of $\vv Q$ is equivalent to all subquasivarieties of $\vv Q$ being equational (that is, every subquasivariety of $\vv Q$ can be defined relative to $\vv Q$ by a set of equations), hereditarily structurally complete quasivarieties also called deductive or \Def{primitive} (see \cite{Bergman1991}), and we use the latter term. It follows that any question about structural completeness  (and its variations, see \cite{AglianoUgolini2023}) of an algebraizable logic can be solved in purely algebraic terms.  On the other hand, given the fact that structural completeness for quasivarieties does not make any reference to logic, the concept makes sense also for quasivarieties that are not quasivarieties of logic.

In \cite{Mints1971}, Mints observed that any admissible in the intuitionistic propositional calculus (IPC) rule which does not contain either connectives $\lor$ or $\to$ is derivable in IPC. In other words, IPC is $\{\land, \lor, \neg\}$- and $\{\land, \to, \neg\}$-structurally complete. This observation leads to the following problem: given a set of formulas $F$ in the language of logic $\vv L$, we say that $\vv L$ is \textbf{structurally complete relative to} $F$ if every admissible rule whose premises and conclusion are in $F$, is derivable in $\vv L$. For instance, one can take $F$ as a set of all formulas of IPC not containing disjunction and IPC will be structurally complete relative to such $F$. Moreover, we say that logic $\vv L$ \textbf{is p-structurally complete relative to} $F$, if every admissible rule whose premises (not necessarily conclusions) are in $F$ is derivable. For instance, IPC is p-structurally complete (and even hereditarily p-structurally complete) relative to set of all Harrop (or all anti-Harrop) formulas (see \cite{Roziere1993}). Another set of formulas relative to which IPC is p-complete was introduced in \cite{Citkin1979}.

This observation led the second author to introducing a notion of $C$-structural completeness (see \cite{Citkin2016}): if $C$ is a subset of the set of connectives of a propositional logic $\vv L$, $\vv L$ is $C$-structurally complete if every admissible rule containing only connectives from $C$ is derivable in $\vv L$. $C$-structural completeness is very close but not the same as structural completeness of the $C$-fragment of $\vv L$. For instance, from Mints' observation it follows that IPC is $\{\to,\neg\}$-structurally complete, while $\{\to,\neg\}$-fragment of IPC is not (see \cite{CintulaMetcalfe2010, Citkin2016} for details). In some cases $C$-structural completeness can be extended to structural completeness; for instance, the proof from \cite{Mints1971} holds for any axiomatic extension of IPC, thus it holds for the Dummett calculus LC = IPC + $((p \to q) \lor (q \to p))$ and because $\lor$ can be expressed in LC by a formula, LC (and every its extension) is structurally complete, and this is an alternative proof of hereditary structural completeness of LC first observed in \cite{DzikWronski1973} (the same argument is used to prove primitivity in \cite{GalatosRaftery2015}, Theorem 2.9).

In the algebraic setting, if $\vv Q$ is a quasivariety and $C$ is a subset of the set of basic operations, then $\vv Q$ is \textbf{$C$-structurally complete} if for every quasiequation $\Phi$ containing operations only from $C$, $\alg F_\vv Q(\omega) \models \Phi$ entails $\vv Q \models \Phi$. In other words, if $\Phi$ is refuted in some algebra from $\vv Q$, then $\Phi$ must be refuted in $\alg F_\vv Q(\omega)$. On the other hand, if $\vv Q^C$ is a quasivariety generated by all $C$-reducts of algebras from $\vv Q$, then $\vv Q^C$ is structurally complete if every quasiequation $\Phi$ is refuted in $\alg F_{\vv Q^C}(\omega)$ as long as it is refuted in some algebra from $\vv Q^C$, or equivalently, in some $C$-reduct of an algebra from $\vv Q$. Let us observe that $C$-reduct of $\alg F_\vv Q(\omega)$ belongs to $\vv Q^C$ and hence, structural completeness of $\vv Q^C$ entails $C$-structural completeness of $\vv Q$, while the converse does not have to be true.

Admissibility and derivability of quasiequation $\Phi = \bigwedge_{i \in [1,n]} t_i \approx s_i \Rightarrow t \approx s$ in a quasivariety $\vv Q$ depend on properties of the algebra $\alg A$ finitely presented by relations $t_i \approx s_i, i \in [1,n]$. Indeed, $\Phi$ is derivable in $\vv Q$, that is $\vv Q \models \Phi$, if and only if the terms $t$ and $s$ represent the same element of $\alg A$. And $\Phi$ is admissible in $\vv Q$ if and only if none of  homomorphic images of $\alg A$ in which $t$ and $s$ represent distinct elements can be embeddable into $\alg F_\vv Q(\omega)$. In particular, if $\alg A$ is subdirectly irreducible and $t$ and $s$ represent elements from the monolith of $\alg A$, $\Phi$ is admissible in Q if and only if $\alg A$ is not embeddable in $\alg F_\vv Q(\omega)$; in this case, $\Phi$ is a characteristic rule of $\alg A$ (cf. \cite{Citkin1977}).

The following simple sufficient condition is often used to demonstrate $C$-structural completeness:  if $C$-reduct of every algebra finitely presented in $\vv Q$ can be embedded in the $C$-reduct of $\alg F_\vv Q(\omega)$, then $\vv Q$ is  $C$-structurally complete. Indeed, if a quasiequation $\Phi = \bigwedge_{i \in [1,n]} t_i \approx s_i \Rightarrow t \approx s$ whose terms contain operations only from $C$ is not valid in $\vv Q$, then there is an algebra $\alg A \in \vv Q$ and a valuation $v$ in $\alg{A}$ such that $v(t_i) = v(s_i)$ for all $i \in [1,n]$ and $v(t) \neq v(s)$. Thus, $\Phi$ can be refuted in algebra $\alg B$ finitely presented in $\vv Q$ and defined by the relations $t_i \approx s_i, i \in [1,n]$;  and if  $\alg B$ can be embedded in $\alg F_\vv Q(\omega)$ as $C$-subreduct, then $\Phi$ can be refuted in $\alg F_\vv Q(\omega)$ as well.

As an example let us take $\vv Q$ to be a quasivariety of Heyting algebras in the signature $\{\to, \land, \lor, \neg\}$ and let $C = \{\to, \land\}$ and $C' = \{\land, \lor\}$ and let us recall that every finitely presented in $\vv Q$ algebra is (isomorphic to) a quotient algebra of some free algebra $\alg F_\vv Q(n)$ by a principal congruence. Let us also recall that for any Heyting algebra $\alg A$ and for every principal congruence $\Th$, each congruence class $a/\Th$ contains a largest and a smallest element. It is a simple exercise to check that the set of all the largest elements of the $\Th$-classes  is closed under $\to$ and $\land$ and hence it forms a $\{\to,\land\}$-subreduct of $\alg A$, while the set of all smallest elements of the $\Th$-classes is closed under $\land$ and $\lor$ and hence it forms a $\{\land,\lor\}$-subreduct of $\alg A$. Thus, the $C$- and $C'$-reducts of any quotient algebra by a principal congruence can be embedded in a $C$- or $C'$-reduct of the preimage. Therefore the $C$- and $C'$-reducts of every finitely presented algebra can be embedded in the respective reducts of a free algebra, which means that $\vv Q$ is $C$- and $C'$-structurally complete.

The paper is structured in the following way: in Section \ref{sec:main} we give the necessary definitions and remind the basic facts about quasivarieties. In Section \ref{sec:wpalgebras} we study relative weakly projective algebras and primitive quasivarieties. In particular we give a criterion of primitivity for tame quasivarieties, that is for quasivarieties in which every finitely generated algebra is finitely presented.  In addition, we prove that every finite subdirectly irreducible algebra is projective in the quasivariety which it generates as long as it is weakly projective in it.

Section \ref{sec:more} is dedicated to studying C-structurally complete and C-primitive quasivarieties. We introduce the notion of u-presentable congruence and we prove that, for quasivariety $\vv Q$ to be $C$-structurally complete, it is sufficient  that every compact $\vv Q$-congruence of $\vv F_\vv Q(\omega)$ is u-presentable. We also give an algebraic generalization of Prucnal's substitution (see \cite{Prucnal1972}).

In Section \ref{sec:idiscriminator} we study in depth dual i-discriminator quasivarieties that are different from the varieties studied in Section \ref{sec:more} but to which one can apply techniques similar to  Prucnal's substitution. We prove several structure theorems, connecting them with the theory of if ideals in quasivarieties, and we show that every relatively congruence distributive dual i-discriminator variety in which every fundamental operation commutes with the i-discriminator is primitive.

In Section \ref{sec:conclusions} we discuss some questions arising during our investigations into the matter of this paper.

\section{Main definitions.}\label{sec:main}
\subsection{Class operators, quasivarieties an free algebras.}
Let $\vv K$ be a class of algebras; we denote by $\II,\HH,\PP,\SU,\PP_u$ the class operators sending $\vv K$ in the class of all isomorphic copies, homomorphic images, direct products, subalgebras and ultraproducts of members of $\vv K$. The operators can be composed in the obvious way; for instance $\SU\PP(\vv K)$ denotes all algebras that are embeddable in a direct product of members of $\vv K$; moreover there are relations among the classes resulting from applying operators in a specific orders, for instance $\PP\SU(\vv K) \sse \SU\PP(\vv K)$ and $\HH\SU\PP(\vv K)$ is the largest class we can obtain composing the operators. We will use all the known relations without further notice, but the reader can consult \cite{Pigozzi1972} or \cite{BurrisSanka} for a textbook treatment. We only point out that if $\vv K$ is a finite set of finite algebras, then $\PP_u(\vv K) = \vv K$.

 If $\rho$ is a type of algebras, an \Def{equation} is a pair $p,q$ of $\rho$-terms (i.e. elements of the absolutely free algebra $\alg T_\rho(\o)$) that we write suggestively as $p \app q$; a \Def{universal sentence}  in $\rho$ is a formal pair $(\Sigma, \Gamma)$ that we write as  $\Sigma \Rightarrow \Gamma$, where $\Sigma,\Gamma$ are finite sets of equations; a universal sentence is  a \Def{quasiequation} if $|\Gamma| = 1$. Clearly an  equation is a quasiequation in which $\Sigma = \emptyset$.

Given any set of variables $X$, an assignment of $X$ into an algebra $A$ of type $\rho$ is a function $h$ mapping each variable $x \in X$ to an element of $\alg A$, that extends (uniquely) to a homomorphism (that we shall also call $h$) from the absolutely free algebra $\alg T_\rho(\o)$ to $\alg A$.
An algebra $\alg A$ satisfies an equation $p \app q$ with an assignment $h$ (and we write $\alg A, h \models p \approx q$) if $h(p) = h(q)$ in $\alg A$.  An equation $p \app q$ is \Def{valid} in $\alg A$ (and we write $\alg A \vDash p \approx q$) if for all assignments $h$ in $\alg A$, $\alg A, h \models p \approx q$; if $\Sigma$ is a set of equations then
 $\alg A \vDash \Sigma$ if $\alg A \vDash \sigma$ for all $\sigma \in \Sigma$.
 A universal sentence is \Def{valid} in $\alg A$ (and we write $\alg A \vDash \Sigma \Rightarrow \Delta$) if for all assignments $h$ to $\alg A$, $h(p) = h(q)$ for all $p \approx q \in \Sigma$ implies that there is an identity $s \app t \in \Delta$ with
 $h(s) = h(t)$; in other words a universal sentence can be understood as the formula $\forall \mathbf x(\bigwedge \Sigma \imp \bigvee \Delta)$.  An equation or a universal sentence  is \Def{ valid} in a class $\vv K$ if it is valid in all  algebras in $\vv K$.

Let $\vv K$ be any class  of similar algebras:
\begin{itemize}
	\item[(a)] if $\vv K = \HH\SU\PP \vv K$, then $\vv K$ is a \Def{variety};
	\item[(b)] if $\vv K = \II\SU\PP\PP_u \vv K$, then $\vv K$ is a \Def{quasivariety};
	\item[(c)] if $\vv K = \II\SU\PP_u \vv K$, then $\vv K$ is a \Def{universal class}.	
\end{itemize}
Of course we will use the standard abbreviations  $\VV$ for $\HH\SU\PP$ and $\QQ$ for $\II\SU\PP\PP_u$; from the properties of the operators
one can easily deduce that for any class $\vv K$, $\QQ(\vv K) \sse \VV(\vv K)$.

The following facts were essentially discovered by A. Tarski , J. \L\`os and A. Lyndon in the pioneering phase of model theory; for proof of this and similar statements the reader can consult \cite{ChangKeisler}.

 \begin{theorem}\label{theorem:ISP}
 Let $\vv K$ be any class of algebras. Then:
 \begin{enumerate}
 \item  $\vv K$ is a universal class  if and only if it is the class of all algebras in which a given set  of universal sentences is valid;
\item  $\vv K$ is a quasivariety  if and only if it is the class of all algebras in which a given set  of quasiequations  is valid;
\item $\vv K$ is a variety  if and only if it is the class of all algebras in which a given set  of equations  is valid.
\end{enumerate}
  \end{theorem}

For the definition of free algebras in a class $\vv K$ on a set $X$ of generators, in symbols $\alg F_\vv K(X)$, we refer to \cite{BurrisSanka}; we merely observe that a class $\vv K$ contains all the free algebras in $\vv K$ if and only if $\vv K = \II\SU\PP(\vv K)$; therefore for any quasivariety $\vv Q$, $\alg F_\vv Q(X) = \alg F_{\VV(\vv Q)}(X)$.

\subsection{Congruences and subdirectly irreducible algebras.}

A \Def{ congruence} on an algebra $\alg A$ is an equivalence relation that is compatible (in the usual sense) with all the fundamental operations of $\alg A$.  As the intersection of any family of congruences on $\alg A$ is again a congruence on $\alg A$, the mapping sending $X\sse A^2$ into $\cg_\alg A(X)$ (the smallest congruence of $\alg A$ containing $X$) is a closure operator on $A^2$ that can be proved too by algebraic. Thus the congruences of $\alg A$ form an algebraic lattice which we denote by $\Con A$; hence there are always a largest and smallest congruence in $\Con A$ which will be denoted by $1_\alg A$ and $0_\alg A$. A variety $\vv V$ is \Def{ congruence distributive} if all congruence lattices of algebras in $\vv V$ are distributive.

Let $\alg B, (\alg A_i)_{i \in I}$ be algebras in the same signature; we say that $\alg B$ {\em embeds} in $\prod_{i \in I} \alg A_i$  if $\alg B \in \II\SU(\prod_{i\in I} \alg A_i)$.
Let $p_i$ be the $i$-th projection, or better, the composition of the embedding and the $i$-th projection, from $\alg B$ to $\alg A_i$; the embedding is
\Def{ subdirect} if for all $i \in I$, $p_i(\alg B) = \alg A_i$ and in this case we will write
$$
\alg B \le_{sd} \prod_{i \in I} \alg A_i.
$$
An algebra $\alg B$ is \Def{subdirectly irreducible} if it is nontrivial and for any subdirect embedding
$$
\alg B \le_{sd} \prod_{i \in I} \alg A_i
$$
there is an $i \in I$ such that $\alg B$ and $\alg A_i$ are isomorphic. An algebra $\alg A$ is {\bf finitely subdirectly irreducible} if the same conclusion as above holds whenever $I$ is finite.

\begin{lemma} For an algebra $\alg A$ the following are equivalent:
\begin{enumerate}
\item $\alg A$ is subdirectly irreducible;
\item $\Con A$ has exactly one minimal congruence $\mu_\alg A > 0_\alg A$ (which of course it is equivalent to saying that $0_\alg A$ is strictly meet irreducible in $\Con A$);
\item there are $a,b \in A$ such that $(a,b) \in \th$ for all $\th \in \Con A$, $\th >0_\alg A$.
\end{enumerate}
\end{lemma}

For any variety $\vv V$ we denote by $\vv V_{fsi}$ the class of finitely subdirectly irreducible members of $\vv V$; we have  the  classical results:

\begin{theorem}\label{birkhoff} \begin{enumerate}
 \item (Birkhoff \cite{Birkhoff1944}) Every algebra can be subdirectly embedded in a product of subdirectly irreducible algebras. So if $\alg A \in \vv V$, then $\alg A$ can be subdirectly embedded in a product of members of $\vv V_{si}$ and then $\VV(\vv V_{si}) = \vv V$.
\item (J\'onsson's Lemma \cite{Jonsson1967}) Suppose that $\vv K$ is a class of algebras such that $\vv V(\vv K)=\vv V$ is congruence distributive;
then $\vv V_{fsi} \sse \HH\SU\PP_u(\vv K)$.
\end{enumerate}
\end{theorem}

\subsection{Relative congruences}

If $\vv Q$ is a quasivariety and $\alg A \in \vv Q$, a \Def{$\vv Q$-congruence} (of $\alg A$ is a congruence $\th$ such that $\alg A/\th \in \vv Q$; as the intersection of any family of $\vv Q$-congruences of $\alg A$ is a $\vv Q$-congruence, they
form an algebraic lattice $\op{Con}_\vv Q(\alg A)$. If  $H \sse A^2$  the smallest   $\vv Q$-congruence of $\alg A$ containing $H$ will be denoted by $\theta_{\vv Q}(H)$; when $H = \{(a,b)\}$, we just write $\theta_{\vv Q}(a,b)$.
If $\alg A \in \vv Q$ we say that $\alg A$ is \Def{relatively subdirectly irreducible} in $\vv Q$ (or briefly \Def{$\vv Q$-irreducible}) if for any subdirect embedding
$$
\alg B \le_{sd} \prod_{i \in I} \alg A_i
$$
for which $\alg A_i \in \vv Q$ for all $i \in I$, there is an $i \in I$ such that $\alg B$ and $\alg A_i$ are isomorphic. The concept of {\bf finitely $\vv Q$-irreducible} is defined in an obvious way.

\begin{lemma}\label{lemma: q-irreducible} Let $\vv Q$ be a quasivariety and $\alg A\in \vv Q$. Then the following are equivalent:
\begin{enumerate}
\item $\alg A$ is $\vv Q$-irreducible;
\item $\op{Con}_\vv Q(\alg A)$ has exactly one minimal congruence above $0_\alg A$;
\item there are $a,b \in A$ such that $(a,b) \in \th$ for all $\th \in \op{Con}_\vv Q (\alg A)$, $\th >0_\alg A$.
\end{enumerate}
\end{lemma}

The minimal congruence mentioned in Lemma \ref{lemma: q-irreducible} is called the \Def{$\vv Q$-monolith} of $\alg A$. Observe also that, since $\op{Con}_\vv Q(\alg A)$ is a meet subsemilattice of $\Con A$, if $\alg A$ is subdirectly irreducible and
$\alg A \in \vv Q$, then $\alg A$ is $\vv Q$-irreducible. If $\vv Q$ is any subquasivariety we denote by $\vv Q_{ir}$  and $\vv Q_{fir}$ the classes of $\vv Q$-irreducible algebras and finitely $\vv Q$-irreducible algebras
 in $\vv Q$, respectively.
We have the equivalent of Birkhoff's and J\'onsson's results for quasivarieties:

\begin{theorem}\label{quasivariety} Let $\vv Q$ be any quasivariety.
\begin{enumerate}
\item (Mal'cev \cite{Malcev1956}) Every $\alg A \in \vv Q$ is subdirectly embeddable in a product of algebras  in $\vv Q_{ir}$.
\item (Czelakowski-Dziobiak \cite{CzelakowskiDziobiak1990}) If $\vv Q = \QQ(\vv K)$, then $\vv Q_{fir} \sse \II\SU\PP_u(\vv K)$.
\end{enumerate}
\end{theorem}

First we observe:

\begin{lemma} \label{lemma: Q(A) variety} Let $\alg A$ be an  algebra, such that $\VV(\alg A)$ is congruence distributive. Then
$\QQ(\alg A) = \VV(\alg A)$ if and only if every subdirectly irreducible algebra in $\HH\SU\PP_u(\alg A)$ is  in $\II\SU\PP_u\alg A$.
\end{lemma}
\begin{proof}
Suppose first that $\QQ(\alg A) = \VV(\alg A)$; since any subdirectly irreducible algebra $\alg B$ in $\HH\SU\PP_u(\alg A)$ is $\vv Q(\alg A)$-irreducible  by Theorem \ref{quasivariety} $\alg B \in \II\SU\PP_u(\alg A)$.

Conversely assume that every subdirectly irreducible algebra in $\HH\SU\PP_u(\alg A)$ is in $\II\SU\PP_u\alg A$. Since $\VV(\alg A)$ is congruence distributive, by Theorem \ref{birkhoff}(2) every subdirectly irreducible algebra in $\VV(\alg A)$ is in $\HH\SU\PP_u(\alg A)$, thus in $\II\SU\PP_u\alg A$. Now every algebra in $\VV(\alg A)$ is subdirectly embeddable in a product of subdirectly irreducible algebras in $\VV(\alg A)$ (Theorem \ref{birkhoff}(1)). Therefore, $\VV(\alg A) \sse \II\SU\PP\II\SU\PP_u(\alg A) \sse \II \SU\PP\PP_u(\alg A) = \QQ(\alg A)$ and thus equality holds.
\end{proof}

Next:

\begin{proposition}\label{pr-maxcon}
	Let $\vv Q$ be a quasivariety, $\alg A \in \vv Q$, and $\alg B \leq \alg A$. Then, there exists a maximal $\vv Q$-congruence $\th$ on $\alg A$ such that
	\begin{align*}
	\th|_\alg B = 0_\alg B.
	\end{align*}
Moreover if $\alg B$ is $\vv Q$-irreducible, then $\alg A/\th$ is $\vv Q$-irreducible.
\end{proposition}
\begin{proof}
	Let $V = \{\a \in \op{Con}_{\vv Q}\alg A: \a|_\alg B = 0_\alg B\}$; since $0_\alg A \in V$, $V \neq \varnothing$.
	Let $C$ be an ascending chain in $V$, and let $\th := \Join C$; 	
	as $\vv {Con}_{\vv Q}\alg A$ is an algebraic lattice, if for some $a,b \in \alg B$, $a \neq b$ and $(a,b) \in\th$, then there is an  $\a \in C$ such that $(a,b) \in \a$, which contradicts $\a \in V$. By Zorn's Lemma
$V$ has a maximal congruence $\th$.
	
	If $\alg B$ is $\vv Q$-irreducible, then (by Lemma \ref{lemma: q-irreducible}(3)) there are two elements $a,b \in \alg B$ such that $(a,b) \in \th'$ for all $\th' >0_\alg B$.  Now a relative proper congruence $\a$ of $\alg A/\th$ is
 of the form $\th'/\th$ for some $\th' > \th$. Since $\th$ is a maximal in $V$  we must have $(a,b) \in \th'$ and thus $(a/\th,b/\th) \in \th'/\th$;  by Lemma \ref{lemma: q-irreducible}(3), $\alg A/\th$ is $\vv Q$-irreducible.
\end{proof}

\subsection{Relatively finitely presented algebras.}

Let $\sigma$ be a type and let $\vv K$ be a class of algebras of type $\sigma$; let $X$ be a set of variables and $\Sigma$ be a set of equations of type $\sigma$ in variables from $X$.  We say that
the pair $(X,\Sigma)$ is \Def{ a $\vv K$-presentation of} $\alg A \in \vv K$ if there exists a function $\a: X \longrightarrow A$ such that
\begin{enumerate}
\ib $\a(X)$ generates $\alg A$ and for any $p(\vuc xn) \app q(\vuc xn) \in \Sigma$,
$$
p(\a(x_1),\dots,\a(x_n)) = q(\a(x_1),\dots,\a(x_n));
$$
\ib if $\alg B \in \vv K$ and $\b: X \longrightarrow B$ such that for any $p(\vuc xn) \app q(\vuc xn) \in \Sigma$, $p(\b(x_1),\dots,\b(x_n)) = q(\b(x_1),\dots,\b(x_n))$, then
there exists a homomorphism $f:\alg A \longrightarrow \alg B$ such that $f(\a(x)) = \b(x)$ for all $x \in X$ (Figure \ref{K-presentation}).
\end{enumerate}

\begin{figure}[htbp]
\begin{center}
\begin{tikzpicture}[scale=.7]
\node[left] at (0,3) {\footnotesize  $\alg X$};
\node[right] at (3,3) {\footnotesize  $\alg A$};
\node[right] at (3,0) {\footnotesize  $\alg B$};
\draw[->] (0,3)-- (3,3);
\draw[->] (0,2.8) -- (3,0.2);
\draw[->][blue] (3.3,2.8) -- (3.3,.2);
\node[above] at (1.5,3) {\footnotesize $\a$};
\node[below] at (1.3,1.6) {\footnotesize $\b$};
\node[right] at (3.3,1.5) {\footnotesize  $f$};
\end{tikzpicture}
\caption{}\label{K-presentation}
\end{center}
\end{figure}

Let $F(X,\Sigma)$ be a representation of  $\vv Q$-irreducible algebra $\alg A$ and let $(p,q)$ any pair of elements generating the $\vv Q$-monolith of $\alg A$.
We call the quasiidentity
$$
\Sigma \Rightarrow p \app q
$$
a \Def{characteristic quasiidentity} of $\alg A$. We denote it by $\op{ch}(\alg A)$ the set of characteristic quasiidentities of $\alg A$.

\begin{lemma}\label{lemma: characteristic} If an algebra $\alg A \in \vv Q$ is $\vv Q$-irreducible  and $\vv Q$-finitely presented, then for every algebra $\alg B \in \vv Q$ and for every $\Phi \in \op{ch}(\alg A)$
$\alg B \not\vDash \Phi \iff \alg A \in \II\SU(\alg B)$.
\end{lemma}
\begin{proof} The right-to-left implication is obvious since $\alg A \not\vDash \Phi$.  Assume then that $\alg B \not\vDash \Phi$ and let
$$
\Phi = \{ p_i\app q_i: i \in I\} \Rightarrow p \app q.
$$
Then there are $\vuc bn \in B$ such that $p_i(\vuc bn) = q_i(\vuc bn)$ but $p(\vuc bn) \ne p(\vuc bn)$. Let $g$ be the homomorphism extending the assignment $x_i \longmapsto b_i$; then $\th_\vv Q(\Sigma) \sse \op{ker}(g)$ so by the Second Homomorphism Theorem there is a homomorphism $f: \alg A \longrightarrow \alg B$ such that $f(a_i) = b_i$.  Observe that  $f(\alg A) \in \vv Q$ (since it is a subalgebra of $\alg B \in \vv Q$) and $f(\alg A) \not\vDash \Phi$; but since the pair $(p,q)$ generates the monolith of $\alg A$ any proper homomorphic image of $\alg A$ must satisfy $\Phi$. Hence $f(\alg A) \cong \alg A$ and thus $\alg A \in \II\SU(\alg B)$.
\end{proof}

Lemma \ref{lemma: characteristic} implies that all the characteristic quasiidentities are interderivable in the Birkhoff's sense, i.e. they define the same quasivariety; therefore the choice of $\Sigma,p,q$ does not matter.
Therefore, if $\vv Q$ is a quasivariety for any $(X,\Sigma)$ there is exactly one algebra (up to isomorphism)  $\alg A \in \vv Q$ that is $\vv Q$-presented by $(X,\Sigma)$; we will denote that algebra
by $\alg F_\vv Q(X,\Sigma)$ and we may assume that $\a(X)=X$, i.e. $\a$ is the identity mapping. Note that  $\alg F_\vv Q(X,\emptyset) = \alg F_\vv Q(X)$. An algebra $\alg A \in \vv Q$ is \Def{ $\vv Q$-finitely presented} if $\alg A \cong \alg F_\vv Q(X,\Sigma)$ for some finite sets $X,\Sigma$.

If $p \app q$ is an equation and $\Sigma$ is a set of equations in variables from $X$, then $p \app q $ is a \Def{$\vv K$-consequence of $\Sigma$} (in symbols, $\Sigma \vdash_\vv K p\app q$), if for every $\alg B \in \vv K$ and every map $\beta: X \longrightarrow \alg B$,
\[
\alg B \models \beta(\Sigma) \text{ entails } \alg B \models \beta(p) \app \beta(q).
\]

\begin{theorem}[{\cite{Malcev1973}, Section 5}]
	Let $\alg A  \in \vv K$ be an algebra $\vv K$-presented by $(X,\Sigma)$ and map $\alpha:X \longrightarrow \alg A$. Then for any terms $p$ and $q$ in variables from $X$,
	\[
	p(\alpha(\vec x)) = q(\alpha(\vec x)) \iff \Sigma \vdash_\vv K p(\vec x) \app q(\vec x).
	\]
\end{theorem}

\begin{corollary} Let $\vv K$ be a class of algebras. Then
\begin{enumerate}
\item if $\Sigma_1$ and $\Sigma_2$ are sets of equations in variables from $X$ and $\Sigma_1 \vdash_\vv K \Sigma_2$, then there is a homomorphism of $\alg F(X,\Sigma_1)$ onto $\alg F(X,\Sigma_2)$;
\item  $(X,\Sigma_1)$ and $(X,\Sigma_2)$ $\vv K$-define the same (up to isomorphism) algebra if and only if
$$
\Sigma_1 \vdash_\vv K \Sigma_2 \text{ and }\Sigma_2 \vdash_\vv K \Sigma_1.
$$
\end{enumerate}
\end{corollary}

Let   $X$ be a set of variables and $\Sigma$ be a set of equations of type $\sigma$ in variables from $X$; we denote by
$\th_\vv Q(\Sigma)$ the $\vv Q$-congruence of $\alg F_\vv Q(X)$ generated by all pairs in $\{(p,q): p \app q \in \Sigma\}$.

\begin{lemma} \label{lemma: equivalent} Let $\vv Q$ be a quasivariety,   $X$ a set of variables and $\Sigma$ a  set of equations in variables from $X$; if $p\app q$ is
an equation in variables from $X$ then the following are equivalent:
\begin{enumerate}
\item $\alg F_\vv Q(X,\Sigma) \vDash p \app q$;
\item $\Sigma \vdash_\vv Q p \app q$;
\item $(p,q) \in \th_\vv Q(\Sigma)$.
\end{enumerate}
\end{lemma}

\begin{theorem} Let $\vv Q$ be a quasivariety. For any $\alg A \in \vv Q$ the following are equivalent:
\begin{enumerate}
\item $\alg A$ is finitely $\vv Q$-presented by $(X,\Sigma)$;
\item $\alg A \cong \alg F_\vv Q(X)/\th_\vv Q(\Sigma)$.
\end{enumerate}
\end{theorem}
\begin{proof} It is enough to prove that $\alg F_\vv Q(X,\Sigma) \cong \alg F_\vv Q(X)/\th_\vv Q(\Sigma)$. Take $\b$ to be the map sending $x \longmapsto x/\th_\vv Q(\Sigma)$; then  for all $p \app q \in \Sigma$
$$
p(\b(\vec x)) = \b(p(\vec x) = p(\vec x)/\th_\vv Q(\Sigma) = \dots= q(\b(\vec x));
$$
Hence by definition there is a homomorphism $f: \alg F_\vv Q(X,\Sigma) \longrightarrow \alg F_\vv Q(\Sigma)/\th_\vv Q(\Sigma)$ with $f(x) = \b(x)$ (where we have taken $\a = id_X$). Since $\{\b(x): x \in X\}$ generates
$\alg F_\vv Q(X)/\th_\vv Q(\Sigma)$, $f$ is onto (see Figure \ref{q-presented}).

\begin{figure}[htbp]
\begin{center}
\begin{tikzpicture}[scale=.7]
\node[left] at (0,3) {\footnotesize  $\alg X$};
\node[right] at (3,3) {\footnotesize  $\alg F_\vv Q(X,\Sigma)$};
\node[right] at (3,0) {\footnotesize  $\alg F_\vv Q/\th_\vv Q(\Sigma)$};
\draw[->] (0,3)-- (3,3);
\draw[->] (0,3) -- (2.9,3);
\draw[->] (0,2.8) -- (3,0.2);
\draw[->][blue] (3.3,2.8) -- (3.3,.2);
\node[above] at (1.5,3) {\footnotesize $id_X$};
\node[below] at (1.3,1.6) {\footnotesize $\b$};
\node[right] at (3.3,1.5) {\footnotesize  $f$};
\end{tikzpicture}
\caption{}\label{q-presented}
\end{center}
\end{figure}

Let now $u,v \in \alg F_\vv Q(X,\Sigma)$ with $f(u)=f(v)$; as $X$ generates the algebra, there are terms $s,t$ and $\vuc xk \in X$ with $u=s(\vuc xk)$ and $v= t(\vuc xk)$.
Hence $s(\f(\vec x)) = s(\f(\vec y))$ i.e. $s(\vec x)/\th_\vv Q(\Sigma) = t(\vec x)/ \th_\vv Q(\Sigma)$, i.e.  $(s,t) \in \th_\vv Q(\Sigma)$. By Lemma \ref{lemma: equivalent},
$\alg F_\vv Q(X,\Sigma) \vDash s \app t$ and thus $u= s(\vec x) = t(\vec x)= v$. Thus $f$ is injective and the conclusion holds.

\end{proof}

 Let $\vv Q$ be a quasivariety and $\alg A \in \vv Q$; we define
$$
[\vv Q:\alg A] = \{\alg B \in \vv Q: \alg A \notin \II\SU(\alg B)\}.
$$

\begin{lemma}\label{lemma: irr fin presented}  Let $\vv Q$ be a  quasivariety of finite type; then
\begin{enumerate}
\item if $\alg A \in \vv Q$ is finite, then $[\vv Q:\alg A]$ is a universal class;
\item if $\alg A \in \vv Q$ is $\vv Q$-irreducible and $\vv Q$-finitely presented, then $[\vv Q:\alg A]$ is a quasivariety.
\end{enumerate}
Moreover if $\vv Q$ is locally finite then the converse implications in (1) and (2) hold.
\end{lemma}
\begin{proof}  For (1), if $\alg A$ is finite, then there is a first order universal sentence $\Psi$ such that, for all $\alg B \in \vv Q$, $\alg B \vDash \Psi$ if and only if
$\alg A \in \II\SU(\alg B)$. More precisely, if $|A|=n$, $\Psi$ is the conjunction of the {\em diagram} of $\alg A$ (which is a conjunction of universal sentences that describe the operation tables
of $\alg A$) and $\bigwedge_{i,j\le n, i \ne j} \neg(x_i \app x_j)$.
Now it is well-known (and easy to check) that $\Psi$ defines the universal class $\vv U$ of algebras of the given type that do not contain $\alg A$ as a subalgebra. Therefore
$[\vv Q;\alg A] =\vv Q \cap \vv U$; as the intersection of two universal classes is an universal class, it follows that $[\vv Q:\alg A]$ is universal.

 Consider $\alg B \in \II\SU\PP_u([\vv Q:\alg A])$, we show that $\alg A \notin \II\SU(\alg B)$; if $\alg A \in \II\SU(\alg B)$, then $\alg A \in \II\SU\PP_u([\vv Q:\alg A])$. Hence there exists a family $(\alg A_i)_{i\in I} \sse [\vv Q: \alg A]$ and
an ultrafilter $U$ on $I$ such that $\alg C = \Pi_{i\in I}\alg A/U$ and $\alg A \in \II\SU(\alg C)$.  So  $\alg C \vDash \Psi$;
 but then by the \L{\`o}s Lemma there is a (necessarily nonempty) set of indexes $I' \in U$ such that  $\Psi$ is valid in each $\alg A_i$ with $i \in I'$, which is clearly a contradiction, since each $\alg A_i \in [\vv Q:\alg A]$.
Thus $\alg A \notin \II\SU(\alg B)$ and $\alg B \in [\vv Q:\alg A]$ and  therefore $\II\SU\PP_u([\vv Q:\alg A]) = [\vv Q:\alg A]$ which is therefore a universal class.

Conversely let $\vv Q$ be locally finite of finite type; every algebra in $\vv Q$ is embeddable in an ultraproduct of  its finitely generated (i.e. finite) subalgebras, say $\alg A \in \II\SU\PP_u(\{\alg B_i: i\in I\})$. If $\alg A$ is not finite, then $\alg A \notin \SU(\alg B_i)$ for all $i$, so $\alg B_i \in [\vv Q:\alg A]$ for all $i$. Since $[\vv Q:\alg A]$ is universal, we would have that $\alg A \in [\vv Q:\alg A]$, a clear contradiction. So $\alg A \in \II\SU(\alg B_i)$ for some $i$ and hence it is finite.

For (2), suppose that $\alg A$ is $\vv Q$-irreducible and finitely presented, i.e.  $\alg A \cong \alg F_\vv Q(\mathbf x)/\th_\vv Q(\Sigma)$.
If $\Phi$ is  a characteristic quasiidentity of $\alg A$  then by Lemma \ref{lemma: characteristic} we get at once that $[\vv Q:\alg A] =\{\alg B \in \vv Q: \alg B \vDash \Phi\}$ and this of course implies that $[\vv Q:\alg A]$ is a quasivariety.

For the converse, let $\vv Q$ be locally finite of finite type; by (1) $\alg A$ is finite. Suppose that $\alg A \le_{sd} \prod_{i \in I} \alg B_i$ where each $\alg B_i$ is $\vv Q$-irreducible in $\vv Q$.
Since $\alg A$ is finite, each $\alg B_i$ can be taken to be finite; if $\alg A \notin \II\SU(\alg B_i)$ for all $i$, then $\alg B_i \in [\vv Q:\alg A]$ for all $i$ and hence, being $[\vv Q:\alg A]$ a quasivariety
we have $\alg A \in [\vv Q:\alg A]$ which is impossible. Hence there is an $i$ such that $\alg A \in \II\SU(\alg B)$, so that $|A| \le |B|$; on the other hand $\alg B \in \HH(\alg A)$, so $|B| \le |A|$. Since everything is finite we have $\alg A =\alg B_i$ and $\alg A$ is $\vv Q$-irreducible.
\end{proof}

In quasivarieties, finitely presented algebras sometimes behave as finite algebras.

\begin{corollary}\label{cor: ISu(K)} Let $\vv Q = \QQ(\vv K)$ and let $\alg A \in \vv Q$ be $\vv Q$-irreducible and $\vv Q$-finitely presented; then $\alg A \in \II\SU(\vv K)$.
\end{corollary}
\begin{proof}  If $\alg A \in \vv Q$ is $\vv Q$-irreducible and $\vv Q$-finitely presented, and $\Phi$ is a characteristic quasiidentity of $\alg A$,
then, as $\Phi$ fails in $\alg A$ it must fail for some $\alg B \in \vv K$. It follows that $\alg B \notin [\vv Q: \alg A]$ which by definition implies $\alg A \in \II\SU(\alg B) \sse \II\SU(\vv K)$.
\end{proof}

\section{Relatively weakly projective algebras and primitive quasivarieties.}\label{sec:wpalgebras}

\subsection{Structurally complete quasivarieties.}

A quasivariety $\vv Q$ is \Def{ structurally complete} if any proper subquasivariety of $\vv Q$ generates a proper subvariety of $\HH(\vv Q)$. The proof of the following can be found in several papers (see for instance
\cite{Bergman1991}, \cite{CabrerMetcalfe2015a} and \cite{AglianoUgolini2023}).

\begin{theorem}\label{structural} (see also \cite{Bergman1991}, \cite{CabrerMetcalfe2015a})  For a quasivariety $\vv Q$ the following are equivalent:
\begin{enumerate}
\item $\vv Q$ is structurally complete;
\item for all quasivarieties $\vv Q'\sse\vv Q$ if $\HH(\vv Q') = \HH(\vv Q)$, then $\vv Q = \vv Q'$;
\item for all $\vv K \sse  \vv Q$ if $\VV(\vv K) = \HH(\vv Q)$, then $\QQ(\vv K) = \vv Q$;
\item $\vv Q = \QQ(\alg F_\vv Q(\o))$;
\item every  $\vv Q$-finitely presented algebra is in $\QQ(\alg F_\vv Q(\o))$.
\end{enumerate}
\end{theorem}

For any quasivariety $\vv Q$, we define the \Def{ structural core of  $\vv Q$}   as the smallest $\vv Q'\sse \vv Q$ such that $\HH(\vv Q) = \HH(\vv Q')$. The structural core
of a quasivariety always exists:

\begin{corollary} For any quasivariety $\vv Q$, $\QQ(\vv F_\vv Q(\o))$ is structurally complete and  it is the structural core of  $\vv Q$.
\end{corollary}
\begin{proof} $\QQ(\vv F_\vv Q(\o))$ is structurally complete by Theorem \ref{structural}; if $\vv Q' \sse \vv Q$ is such that $\HH(\vv Q') = \HH(\vv Q)$, then clearly $\alg F_\vv Q(\o) \in \vv Q'$ from which the thesis follows.
\end{proof}

It follows at once that a quasivariety $\vv Q$ is structurally complete if and only if it coincides with its structural core. As a consequence the structurally complete subvarieties of a quasivariety $\vv Q$ are exactly those that coincide with the structural cores of  $\vv Q'$ for some $\vv Q'\sse \vv Q$; even more, since $\HH(\vv Q)$ is a variety, the structurally complete subquasivarieties of a variety $\vv V$ are exactly the structural cores of $\vv V'$ for some subvariety $\vv V'$ of $\vv V$. This observation is particularly useful when  the free countably generated algebra in $\vv V$ has a reasonable description; it has been exploited in \cite{Gispert2016} (for Wajsberg algebras), \cite{Agliano2023} and \cite{AglianoManfucci2023} (for Wajsberg algebras and hoops).

If $\vv Q$ is a quasivariety an algebra $\alg A$ is \Def{ $\vv Q$-exact} if $\alg A \in \II\SU(\alg F_\vv  Q(\o))$. Here are some sufficient conditions for structural completeness.

\begin{lemma}\label{lemma: wpstructcomplete}  Let $\vv Q$ be a quasivariety; if
\begin{enumerate}
\item  $\vv Q= \QQ(\vv K)$ and each $\alg A\in \vv K$ is $\vv Q$-exact, or
\item every finitely generated algebra in $\vv Q$ is $\vv Q$-exact, or
\item every $\vv Q$-finitely presented algebra in $\vv Q$ is $\vv Q$-exact, or
\item every  finitely generated $\vv Q$-irreducible  algebra in  $\vv Q$ is $\vv Q$-exact,
\end{enumerate}
then $\vv Q$ is structurally complete. Moreover if every $\alg A\in \vv K$ is  exact in  $\VV(\vv K)$ and every subdirectly irreducible member of $\VV(\vv K)$ is in $\II\SU(\vv K)$, then $\VV(\vv K)$ is structurally complete.
\end{lemma}
\begin{proof} If each algebra in $\vv K$ is exact in $\vv Q = \QQ(\vv K)$, then $\vv K \sse\II\SU(\alg F_\vv Q(\o))$; therefore $\vv Q =\QQ(\vv K) \sse \QQ(\alg F_\vv Q(\o))$ and thus equality holds. Hence $\vv Q$ is structurally complete by Theorem \ref{structural} . The other points follow from (1) since $\vv Q$ is generated as a quasivariety by all the classes in (2), (3) and (4) (for (3) a proof is in \cite{Gorbunov1998}, Proposition 2.2.18).

 For the last claim, every subdirectly irreducible member of $\VV(\vv K)$ lies in $\II\SU(\vv K)$ and thus is exact in $\VV(\vv K)$. Since any variety is generated as a quasivariety by its subdirectly irreducible members, $\VV(\vv K)$ is structurally complete.
\end{proof}

If $\vv Q$ is locally finite then we get an equivalent condition.

\begin{theorem} For a locally finite variety $\vv Q$ of finite type the following are equivalent:
\begin{enumerate}
\item $\vv Q$ is structurally complete;
\item every finite $\vv Q$-irreducible is $\vv Q$-exact.
\end{enumerate}
\end{theorem}
\begin{proof}  (1) implies (2) (regardless of local finiteness) by  By Corollary \ref{cor: ISu(K)} and Theorem \ref{structural}. If we assume (2), we observe that every locally finite quasivariety is generated as a quasivariety by its finite algebras (as every finitely generated algebra is finite); since they are all in $\II \SU(\alg F_\vv Q(\o))$, we conclude that $\vv Q = \QQ(\alg F_\vv Q(\o))$, i.e. $\vv Q$ is structurally complete.
\end{proof}

\subsection{Primitive quasivarieties.}

Let $\vv Q$ be a quasivariety and $\vv K \sse \vv Q$; we say that $\vv K$ is \Def{ equational relative to} $\vv Q$ if $\vv K = \VV(\vv K) \cap \vv Q$.
Clearly if  $\vv Q, \vv Q'$ are quasivarieties with $\vv Q'\sse Q$ then  $\vv Q'$ is \Def{ equational relative to} $\vv Q$ if and only if $\vv Q' = \HH(\vv Q') \cap \vv Q$; this implies that $\vv Q'$ is axiomatizable modulo $\vv Q$ by a set of equations.

\begin{lemma} Let $\vv Q'$ be equational relative to $\vv Q$. Then:
\begin{enumerate}
\item for every $\alg A \in \vv Q'$, $\op{Con}_{\vv Q'}(\alg A) = \op{Con}_{\vv Q}(\alg A)$;
\item $Q'_{ir} = \vv Q_{ir} \cap \vv Q'$.
\end{enumerate}
\end{lemma}

In particular if all the $\vv Q$-congruence lattices of algebras in  $\vv Q$ satisfy some lattice equation, then the same is true for all the $\vv Q'$-congruence lattices of algebras in any equational subquasivariety
of $\vv Q$.

A quasivariety $\vv Q$ is \Def{ primitive} if every subquasivariety of $\vv Q$ is equational in $\vv Q$; an algebra $\alg A$ is \Def{ weakly $\vv Q$-primitive} \cite{Bergman1991}
if for any algebra $\alg B \in \vv Q$, if $\alg A \in \HH(\alg B)$, then $\alg A \in \II\SU\PP_u(\alg B)$.

\begin{theorem}\label{primitiveQ} For a quasivariety $\vv Q$ the following are equivalent:
\begin{enumerate}
\item $\vv Q$ is primitive;
\item every subquasivariety of $\vv Q$ is structurally complete;
\item every  $\vv Q$-irreducible $\alg A \in \vv Q$ is weakly $\vv Q$-primitive.
\end{enumerate}
\end{theorem}
\begin{proof} We first show the equivalence between (1) and (2). Suppose that $\vv Q$ is primitive and let $\vv Q'\sse \vv Q$; if $\vv Q'' \sse \vv Q'$ and $\HH(\vv Q'') = \HH(\vv Q')$ then
$$
\vv Q' = \HH(\vv Q') \cap \vv Q = \HH(\vv Q'') \cap \vv Q = \vv Q''
$$
so $\vv Q'$ is structurally complete by Theorem \ref{structural}.

Conversely assume (2), let $\vv Q' \sse \vv Q$ and let $\vv Q'' = \HH(\vv Q') \cap \vv Q$ (it is clearly a quasivariety); then $\HH(\vv Q'') = \HH(\vv Q')$ and thus $\vv Q'' = \vv Q'$, again using the characterization of Theorem \ref{structural}. So
$\vv Q'$ is equational in $\vv Q$ and $\vv Q$ is primitive.

Assume (1) again, and let $\alg A,\alg B \in \vv Q$ with $\alg A$ subdirectly irreducible and $\alg A \in \HH(\alg B)$. Since $\vv Q$ is primitive we have
$$
\QQ(\alg B) = \HH(\QQ(\alg B)) \cap \vv Q
$$
and hence $\alg A \in \QQ(\alg B)$. Since $\alg A$ is subdirectly irreducible, $\alg A \in \II\SU\PP_u(\alg B)$ by Theorem \ref{quasivariety} and (3) holds.

 Conversely, assume (3) and let $\vv Q'$ be a subquasivariety of $\vv Q$.  Let $\alg B \in \HH(\vv Q') \cap \vv Q$; then
$\alg B$ is subdirectly embeddable in $\prod_{i\in I} \alg A_i$ where $\alg A_i$ is $\vv Q$-irreducible.
Now for any $i$ $\alg A_i \in \HH(\alg B)$, so $\alg A_i \in \II\SU\PP_u(\alg B)$.
As $\alg B \in \HH(\vv Q')$, there is a $\alg C \in Q'$ with $\alg B \in \HH(\alg C)$ and thus $\alg A_i \in \HH\SU\PP_u(\alg C)$.
So there is a $\alg D_i \in \II\SU\PP_u(\alg C) \sse \vv Q'$ with $\alg A_i \in \HH(\alg D_i)$.  So $\alg A_i \in \II\SU\PP_u(\alg D_i) \sse \vv Q'$
for all $i \in I$; this implies $\alg B \in Q'$, $\vv Q'= \HH(\vv Q') \cap \vv Q$ and $\vv Q$ is primitive.
\end{proof}

We observe that the equivalence of (1) and (3) entails Theorem 2.12 from \cite{Bergman1991}. If instead of a quasivariety we consider a variety we get a little improvement.

\begin{corollary}\label{primvarsc} For a variety $\vv V$ the following are equivalent:
\begin{enumerate}
\item $\vv V$ is primitive;
\item every subquasivariety of $\vv V$ is a variety;
\item $\vv V$ is structurally complete and every proper subvariety of $\vv V$ is primitive.
\end{enumerate}
\end{corollary}
\begin{proof} Trivially (1) and (2) are equivalent and imply (3). Assume then (3) and let $\vv Q$ be a subquasivariety of $\vv V$; if $\vv Q=\vv V$ then it is a variety.
Otherwise $\HH(\vv Q)$ must be a proper subvariety of $\vv V$, so it is primitive; hence $\vv Q$ is a variety and (2) follows.
\end{proof}

We will see later that the hypothesis of structural completeness for $\vv V$ cannot be removed: there are varieties such that every proper subvariety is primitive but fail to be structurally complete.
To get more information we need some definitions: let $\alg A$ be an algebra and $\vv K$ a class of algebras of the same type as $\alg A$. We say that:

\begin{enumerate}
\ib $\alg A$ is \Def{ projective} in $\vv K$ if for all  $\alg B \in \vv K$ if $f: \alg B\longrightarrow \alg A$ is a surjective epimorphism, then there is an embedding $g: \alg A \longrightarrow \alg B$ with $gf=id_\alg A$;
\ib $\alg A$ is \Def{ weakly $\vv K$-projective}  if for all $\alg B\in \vv K$ if $\alg A \in \HH(\alg B)$, then $\alg A \in \II\SU(\alg B)$.
\end{enumerate}
It is clear that if $\alg A$ is projective in $\vv K$, then $\alg A$ is weakly $\vv K$-projective.

\begin{lemma}\label{lemma: techlemma2} Let $\vv Q$ be a quasivariety. Then for $\alg A \in \vv Q$  the following are equivalent:
\begin{enumerate}
 \item $\alg A$ is weakly $\vv Q$-projective;
 \item  $[\vv Q:\alg A]$ is closed under $\HH$.
 \end{enumerate}
\end{lemma}
\begin{proof}
Assume (1), and suppose that $\alg B \in \HH([\vv Q:\alg A])$. If $\alg A \in \II\SU(\alg B)$, then $\alg A \in \SU\HH([\vv Q:\alg A]) \sse \HH\SU([\vv Q:\alg A])$. Now $[\vv Q: \alg A]\sse \vv Q$ and $\alg A$ is weakly $\vv Q$-projective; so $\alg A \in \SU([\vv Q:\alg A])$ which is impossible.  It follows that $\alg A \notin \II\SU(\alg B)$ and $\alg B \in [\vv Q:\alg A]$; thus $[\vv Q:\alg A]$ is closed under $\HH$.
Assume now (2); we show that $\alg A$ is weakly $\vv Q$-projective. Suppose that $\alg A \in \HH(\alg B)$ for some $\alg B \in \vv Q$; if $\alg A \notin \II\SU(\alg B)$, then
$\alg B \in [\vv Q:\alg A]$  and, since $[\vv Q:\alg A]$ is closed under $\HH$, $\alg A \in [\vv Q:\alg A]$, again a contradiction. Hence $\alg A \in \II\SU(\alg B)$ and $\alg A$ is weakly $\vv Q$-projective.
\end{proof}

\begin{corollary}\label{cor: techlemma2} Let $\vv Q$ be a quasivariety; if $\alg A$ is $\vv Q$-irreducible and $\vv Q$-finitely presented then the following are equivalent:
\begin{enumerate}
\item $\alg A$ is weakly $\vv Q$-projective;
\item $[\vv Q: \alg A]$ is $\vv Q$-equational;
\item $[\vv Q:\alg A]$  is a variety.
\end{enumerate}
\end{corollary}
\begin{proof} If (1) holds then $[\vv Q: \vv A]$ is a quasivariety by Lemma \ref{lemma: irr fin presented} and it is closed under $\HH$ by Lemma \ref{lemma: techlemma2}; therefore it is a variety and (3) holds.
Clearly (3) implies (2), so let's assume (2) i.e. $[\vv Q:\alg A] = \HH([\vv Q:\alg A]) \cap \vv Q$.  Let $\alg B \in \vv Q$ with $\alg A \in \HH(\alg B)$; if $\alg A \notin \SU(\alg B)$, then
$\alg B \in [\vv Q:\alg A]$ and thus $\alg A \in \HH([\vv Q:\alg A])$.  As $\alg A \in \vv Q$ we get $\alg A \in [\vv Q:\alg A]$, a contradiction; so $\alg A \in \SU(\alg B)$ and therefore it is weakly $\vv Q$-projective.
\end{proof}

Again Corollary \ref{cor: techlemma2} entails Lemma 5.1.23 from \cite{Gorbunov1998};
also  we get at once that if $\vv Q$ is primitive, then every finitely $\vv Q$-presented $\vv Q$-irreducible algebra  in $\vv Q$ must be weakly $\vv Q$-projective.
The converse however does not seem to be true in general; it is however when we restrict to locally finite quasivarieties (Theorem \ref{mainprimitive} below).

A class $\vv K$ of algebras is \Def{ tame} if every finitely generated algebra in $\vv K$ is $\vv K$-finitely
presented. Note that the concept has content: any class $\vv K$ of algebras of finite type which is
locally finite in the usual sense (i.e. every finitely generated algebra in $\vv K$ is finite)
is tame since in that case the class of finite algebras, the class of finitely generated and the class of finitely presented algebras coincide.
Tame classes of algebras have been studied mainly in groups (better, in algebras in
which groups are interpretable): for instance nilpotent  groups are tame,
so Abelian groups are tame (and it is an example of a tame non locally finite variety).

The next result is implicit in \cite{Gorbunov1998} (Section 5.1) for locally finite quasivarieties (for a self-contained proof the reader can look at \cite{AglianoUgolini2023}, Theorem 4.11); tameness
is weaker than local finiteness but the same proof goes through with minimal changes.

\begin{theorem}\label{mainprimitive} If $\vv Q$ is a tame quasivariety of finite type, then the following are equivalent.
\begin{enumerate}
\item $\vv Q$ is primitive;
\item for all finitely generated  $\alg A \in \vv Q$,  $[\vv Q:\alg A]$ is equational relative to $\vv Q$;
\item every finitely generated $\vv Q$-irreducible  $\alg A \in \vv Q$ is weakly $\vv Q$-projective;
\item every finitely generated  $\vv Q$-irreducible  $\alg A \in \vv Q$ is weakly $\vv Q_{fg}$-projective, where $\vv Q_{fg}$ is the class of finitely generated algebras in $\vv Q$.
\end{enumerate}
\end{theorem}

We observe that (4) implies (1) even in the absence of tameness; so we get:

\begin{corollary} \label{cor: mainprimitive}Let $\vv Q$ be a quasivariety; if every finitely generated $\vv Q$-irreducible algebra is weakly $\vv Q$-projective (or even weakly $\vv Q_{fg}$-projective) then $\vv Q$ is primitive.
\end{corollary}

From the fact that any locally finite variety is tame we get:

\begin{corollary} Let $\vv Q$ be a quasivariety such that every finite $\alg A \in \vv Q$ is weakly $\vv Q$-projective. Then every locally finite subquasivariety of $\vv Q$ is primitive.
\end{corollary}

We have another characterization of weak $\vv Q$-projectivity that gives rise to an interesting class of examples.

\begin{lemma}\label{prucnal} Let $\vv K$ be any class of algebras. Then   the following are equivalent:
\begin{enumerate}
\item every algebra in $\vv K$ is weakly $\vv K$-projective;
\item for all $\alg A \in \vv K$ and for any $\th \in \op{Con}_\vv K(\alg A)$ there is and endomorphism $f$ of $\alg A$ with $\th = \op{ker}(f)$.
 \end{enumerate}
 Moreover every algebra in $\vv K$ is projective in $\vv K$ if and only if $f$ in (2) can always  taken to be idempotent.
\end{lemma}
\begin{proof}  Assume (1) and let $\alg B \in \vv K$ and $\th \in \op{Con}_\vv K(\alg B)$; then $\alg A= \alg B/\th \in \vv K$ and $\alg A \in \HH(\alg B)$.
 By (1), $\alg A \in \II\SU(\alg B)$; if $f$ is the canonical epimorphism and $g$ is the embedding, then $gf$ is and endomorphism of $\alg A$ whose kernel is equal to $\th$. Moreover if $\alg B$ is projective, then $gf$ is the identity and so it is idempotent.

Conversely assume (2) let $\alg A \in \vv K$ such that $\alg A \in \HH(\alg B)$ for some $\alg B \in \vv K$.  Then $\alg A \cong \alg B/\th$ for some $\th \in \op{Con}_\vv K(\alg B)$ and there is an endomorphism $f$ of $\alg B$ with $\op{ker}(f) =\th$.  If $f(\alg B) = \alg C$ then
$$
\alg A \cong \alg B/\th \cong \alg B/\op{ker}(f) \cong \alg C \le \alg B.
$$
So $\alg A \in \II\SU(\alg B)$ and it is weakly $\vv K$-projective. If $f$ is idempotent then for all $b \in B$, $(b,f(b)) \in \op{ker}(f)$; if $g: \alg B \longrightarrow \alg A$ is the canonical epimorphism  then $g(b) = g(f(b))$. If $h: \alg C \longrightarrow \alg A$ is the isomorphism, then $h$ is a monomorphism from $\alg A$ to $\alg B$. If $a \in A$ and $b \in B$ with $g(b)=a$ then
$$
g(h(a))= g(f(b)) = g(b) = a.
$$
We have just proved that $\alg A$ is a retract of $\alg B$, so $\alg A$ is projective in $\vv K$.
\end{proof}

Since every locally finite variety is tame from Theorem \ref{mainprimitive} we get:

\begin{corollary}\label{cor: prucnal}  For any variety $\vv Q$ the following are equivalent:
\begin{enumerate}
\item every finite algebra in $\vv Q$ is weakly $\vv Q_{fin}$-projective, where $\vv Q_{fin}$ is the class of finite algebras in $\vv Q$;
\item every locally finite subquasivariety of $\vv Q$ is primitive.
\end{enumerate}
\end{corollary}

Finally we observe that many of the known examples above have the property that every $\vv Q$-finitely presented $\vv Q$-irreducible algebra  is projective in $\vv Q$, which is an (apparently) stronger condition than the one requested by Theorem \ref{mainprimitive}. This is because it seems plausible that being projective in a quasivariety $\vv Q$ is a strictly stronger concept that being weakly $\vv Q$-projective (and indeed it is, see Example \ref{carr} below). However:

\begin{theorem}\label{thm: wp implies p} Let $\alg A$ be a finite subdirectly irreducible algebra; if $\alg A$ is weakly projective in $\QQ(\alg A)$, then it is projective in $\QQ(\alg A)$.
\end{theorem}
\begin{proof} Let $\vv Q= \QQ(\alg A)$; since $\alg A$ is finite, $\vv Q$ is locally finite. Let $\alg F$ be a finitely generated (hence finite) free algebra in $\vv Q$ such that $\alg A \in \HH(\alg F)$; since $\alg A$ is weakly projective, $\alg A$ is embeddable in $\alg F$ and without loss of generality we may assume that $\alg A \le \alg F$.
Consider the set
$$
 V =\{\a \in \op{Con}_\vv Q(\alg F):  \a \cap A^2 = 0_\alg A\},
$$
where we denote by $0_{\alg A}$ the minimal congruence of $\alg A$.
It is easy to see that $V$ is an inductive poset so we may apply Zorn's Lemma to find a maximal congruence $\th \in V$.  Clearly $a \longmapsto a/\th$ is an embedding of $\alg A$ into $\alg F/\th$. We claim that $\alg F/\th$ is $\vv Q$-irreducible and to prove so, since everything is finite, it is enough to show that $\th$ is meet irreducible in $\op{Con}_\vv Q(\alg F)$;
so let $\a,\b \in \op{Con}_\vv Q(\alg A)$ such that $\a \meet \beta = \th$.  Then
$$
0_ \alg A = \th \cap A^2 = (\a \meet \b) \cap A^2 = (\a \cap A^2) \meet (\b \cap A^2);
$$
But $\alg A$ is subdirectly irreducible, so $0_A$ is meet irreducible in $\op{Con}(\alg A)$; hence either $\a \cap A^2 = 0_\alg A$ or $\b \cap A^2 = 0_\alg A$, so either $\a \in V$ or $\b \in V$. Since $\th$ is maximal in $V$,
either $\a = \th$ or $\b = \th$, which proves that $\alg F/\th$ is relative subdirectly irreducible. Therefore, by Theorem \ref{quasivariety}(2), $\alg F/\th \in \II\SU(\alg A)$; since $\alg F/\th$ and $\alg A$ are both finite and
each one is embeddable in the other, they are in fact isomorphic. Thus $\alg A \le \alg F$, and there is a homomorphism from $\alg F$ onto $\alg A$ that maps each $a \in A$ to itself. This shows that $\alg A$ is a retract of $\alg F$, and therefore $\alg A$ is projective in $\QQ(\alg A)$.
\end{proof}

We close this section with several examples

\begin{example} (Fragments of Heyting algebras)
An early example of an application of Lemma \ref{prucnal}  is in the seminal papers  \cite{Prucnal1972}, \cite{Prucnal1973} and \cite{Prucnal1983}.
Remember that for any Hilbert algebra  $\alg A$  and $a,b,c \in A$
$$
a \imp (b \imp c) = (a \imp b) \imp (a \imp c).
$$
Let now $\alg A$ be a finite Hilbert algebra, $\th \in \Con A$ and $F =1/\th$  the filter associated with
$\th$. Since $\alg A$ is finite, so is $F$ and hence we may assume that $F= \{\vuc an\}$; define
$$
\f(x) = a_1 \imp (a_2 \imp \dots \imp(a_n \imp x)).
$$
Then clearly $\f$ is an endomorphism of $\alg A$; moreover if $\f(b) =1 $ then
$$
a_1 \imp (a_2 \imp \dots \imp(a_n \imp b)) =1
$$
and since $\vuc an \in F$ we get by modus ponens that $b \in F$. Hence $\op{Ker}(\f) \sse F$; on the other hand  $\op{Ker}(\f)$ is a filter
and for any $i \le n$
\begin{align*}
\f(a_i) &= a_1 \imp (a_2 \imp \dots \imp(a_n \imp a_i))\\
&\ge a_1 \imp (a_2 \imp \dots \imp(a_{n-1} \imp a_i))\\
&\quad\vdots\\
&\ge a_1 \imp (a_2 \imp \dots \imp(a_i \imp a_i)) =1
\end{align*}
so $a_1 \in \op{Ker}(\f)$ for $i\le n$. Hence $F=\op{Ker}(\f)$ and every finite Hilbert algebra is projective (since $\f$ is clearly idempotent) in the class of finite Hilbert algebras.
It follows that any locally finite subvariety of Hilbert algebras is primitive; but Hilbert algebras themselves are
locally finite \cite{Diego1966} so any variety of Hilbert algebras is primitive (and so are their prelinear versions, i.e. varieties of G\"odel BCK-algebras)).
All this follows of course also from Prucnal's results (\cite{Prucnal1973}, \cite{Prucnal1983}).

The ${\imp,\meet}$ fragment (i.e.  Brouwerian semilattices) is easier since $(\imp,\meet)$  form a residuated pair. Let $\alg A$ be a Brouwerian semilattice, then
for all $a,b,c \in A$
$$
a \imp (b\meet c) = (a \imp b) \meet (a \imp c).
$$
But Brouwerian semilattices are locally finite \cite{Kohler1981}, so every the variety of Brouwerian semilattices is primitive (and so is its  prelinear version, i.e. varieties of G\"odel hoops \cite{AFM}).

If we add the join (i.e. we have Brouwerian lattices), then we run into problems since we can no longer prove that the map we want to define is an endomorphism.
However in \cite{Citkin2020} the second author proved that there are two Brouwerian lattices, called $\alg S_1$ and $\alg S_2$ such that, for any variety $\vv V$ of Brouwerian lattices
 \begin{enumerate}
 \ib if $\alg S_1 \notin \vv V$, then $\vv V$ is locally finite;
 \ib if $\alg S_1 ,\alg S_2 \notin \vv V$, then every finite Brouwerian lattice in $\vv V$ is projective in $\vv V_{fin}$;
 \ib if either $\alg S_1 \in \vv V$ or $\alg S_2 \in \vv V$, then $\vv V$ is not primitive.
 \end{enumerate}
It follows that a variety of Brouwerian lattices is primitive if and only if it does not contain  $\alg S_1$ and $\alg S_2$; moreover
all such varieties are locally finite.

Primitive subvarieties of Heyting algebras have been totally characterized in \cite{Citkin2018}.
Other applications of Lemma \ref{prucnal} to (fragments of) commutative residuated lattices can be found in \cite{AglianoUgolini2022}; a more general theory encompassing these results is in Section \ref{prucnalterms} below.\qed
\end{example}

\begin{example} \label{carr} The only examples we know of a algebras that are $\vv Q$-weakly projective but not projective appeared very recently.
	 A {\em modal algebra} is a Boolean algebra with a modal operator $\Box$, that we take as a fundamental unary operation, satisfying $\Box 1 \app 1$ and $\Box (x \meet y) \app \Box x \meet \Box y$; there is an extensive literature on modal algebras (see for instance \cite{Wolter1997} and the bibliography therein).
  A modal algebra is a {\em K4-algebra} if it satisfies
 $\Box x \le \Box\Box x$;
 in \cite{Rybakov1995} V.V. Rybakov classified all the primitive varieties of K4-algebras. However very recently \cite{Carr2022} J. Carr discovered a mistake in Rybakov's proof; in his description some varieties are missing. Upon reading the thesis the first author of this paper suggested that this probably depended on the fact that these varieties contained a weakly projective algebra that is not projective. And J. Carr proved exactly that (private communication): all the varieties missed by Rybakov  contain a  weakly projective algebra that is not projective.\qed
\end{example}

\begin{example} The first example of a locally finite variety that is structurally complete but not primitive, appeared  in the seminal paper \cite{Bergman1991}. The idea is to find a finite algebra $\alg A$ such that
$\alg A$ is projective in $\QQ(\alg A)$ (so that $\QQ(\alg A)$ is structurally complete by Lemma \ref{lemma: wpstructcomplete}) but such that $\QQ(\alg A)$ contains a {\em totally non projective} algebra in the sense of  \cite{Citkin2018}.  In \cite{Bergman1991} Bergman  observed that the {\em Fano lattice} $\alg F$ has exactly those characteristics; the Fano lattice is the  lattice of subspaces of $(\mathbb Z_2)^3$ seen as a vector space on $\mathbb Z_2$ and it is displayed in Figure \ref{fano}.

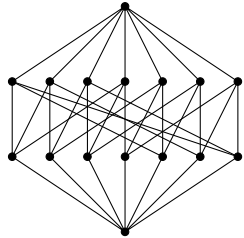
\begin{figure}[htbp]
\begin{center}
\begin{tikzpicture}[scale=1]
\draw (0,0) -- (-1.5,1) -- (-1.5,2) -- (0,3) -- (-1,2) -- (-1,1) -- (0,0) -- (-.5,1) -- (-.5,2) -- (0,3) -- (0,2) -- (0,1) -- (0,0) ;
\draw (0,0) -- (1.5,1) -- (1.5,2) -- (0,3) -- (1,2) -- (1,1) -- (0,0) -- (.5,1) -- (.5,2) -- (0,3)  ;
\draw (-1.5,2) -- (.5,1) -- (1,2);
\draw (-1.5,2) -- (1.5,1);
\draw (-1.5,1) -- (-1,2) -- (1,1) -- (1.5,2);
\draw (-1.5,1) -- (0,2) -- (-.5,1);
\draw (-1,1) -- (-.5,2) -- (1.5,1);
\draw (-1,1) -- (.5,2) -- (0,1);
\draw (-.5,1) -- (1,2);
\draw (0,1) -- (1.5,2);
\draw[fill] (0,0) circle [radius=0.05];
\draw[fill] (0,1) circle [radius=0.05];
\draw[fill] (-.5,1) circle [radius=0.05];
\draw[fill] (-1,1) circle [radius=0.05];
\draw[fill] (-1.5,1) circle [radius=0.05];
\draw[fill] (0.5,1) circle [radius=0.05];
\draw[fill] (1,1) circle [radius=0.05];
\draw[fill] (1.5,1) circle [radius=0.05];
\draw[fill] (0,2) circle [radius=0.05];
\draw[fill] (-.5,2) circle [radius=0.05];
\draw[fill] (-1,2) circle [radius=0.05];
\draw[fill] (-1.5,2) circle [radius=0.05];
\draw[fill] (0.5,2) circle [radius=0.05];
\draw[fill] (1,2) circle [radius=0.05];
\draw[fill] (1.5,2) circle [radius=0.05];
\draw[fill] (0,3) circle [radius=0.05];
\end{tikzpicture}
\caption{The Fano lattice}\label{fano}
\end{center}
\end{figure}
Now  $\alg F$ is projective in $\VV(\alg F)$ \cite{HermannHuhn1976}, hence in $\QQ(\alg F)$, but $\alg F$ has a subalgebra that is totally non projective.  Therefore $\QQ(\alg F)$ is a locally finite structurally complete quasivariety of lattices that is not primitive.

The second author in \cite{Citkin2018} used the same technique to show that there is a variety of Heyting algebras (which is the equivalent algebraic semantics of {\em Medvedev's Logic}) that is structurally complete but not primitive; more examples of structurally complete not primitive subvarieties of this variety can be found in \cite{Skvortsov1998}.  Various examples of structurally complete non primitive  quasivarieties of Wajsberg algebras and Wajsberg hoops appear in  \cite{Agliano2023} and \cite{AglianoManfucci2023}.\qed
\end{example}

\begin{example}\label{abgroups} (Quasivarieties of abelian groups and modules) Since abelian groups are tame we can deal only with finitely generated abelian groups; the Fundamental Theorem  states that a finitely generated abelian group is isomorphic to the direct product (direct sum really, but since there are only finitely many summands it is the same) of a finite power of $\mathbb Z$ and finitely many finite cyclic groups. This implies at once that:
 \begin{enumerate}
 \ib the variety of abelian groups is  $\VV(\mathbb Z)$ (since all finitely generated abelian groups belong to $\VV(\mathbb Z)$);
 \ib a subquasivariety of  abelian groups is proper if and only if it does not  contain $\mathbb Z$;
 \ib every proper subvariety is  finitely generated by finitely many cyclic groups.
 \end{enumerate}
Let's justify the last argument; if $\vv V$ is locally finite let $\alg F$ be the 1-generated free algebra in $\vv V$; then $\alg F$ is finite and cyclic and every cyclic group in $\vv V$ is a homomorphic image of $\alg F$.
As $\alg F$ is finite there are only finitely many of them and, by the Fundamental Theorem, they generate all the finitely generated groups in $\vv V$. Thus $\vv V$ is generated by them and thus it is finitely generated.
Hence every proper subvariety of the variety of abelian groups is locally finite.

Moreover any locally finite variety of abelian groups is primitive; in fact every  finite subdirectly irreducible abelian group is  cyclic of order $p^n$ for some prime $p$.
If  $\vv V$ is a locally finite variety of abelian groups, let  $\alg A \in \vv V$ and let $f: \alg A \longrightarrow \mathbb Z_{p^n}$ be a onto homomorphism; if $b$ is generator of $\mathbb Z_{p^n}$ then any $a \in f^{-1}(b)$ generates in $\alg A$ a finite cyclic group $\alg C$ of which $\mathbb Z_{p^n}$ is a homomorphic image. But it is well-known that homomorphic images and subalgebras of a cyclic group coincide so really  $\mathbb Z_{p^n}\in \II\SU(\alg C)$. Hence $\mathbb Z_{p^n}\in \II\SU(\alg B)$ and it is weakly projective in $\vv V$. By Theorem \ref{mainprimitive} $\vv V$ is primitive.

Now observe that $\mathbb Z$ is {\em torsion-free} so it satisfies all the quasi equations
\begin{equation}
nx \app 0 \qquad \Longrightarrow \qquad x \app 0 \tag{$\e_n$}.
\end{equation}
Hence $\QQ(\mathbb Z)$ satisfies the same quasiequations and  it is a proper subquasivariety; but of course $\HH(\QQ(\mathbb Z))$ is the entire variety of abelian groups, and so the variety of abelian groups cannot be structurally complete. So abelian groups form a variety in which every proper subvariety is primitive but the variety fails to be structurally complete.  To close the circle we observe that also $\QQ(\mathbb Z)$ is primitive since it is an atom in the lattice of subquasivarieties of abelian groups. To prove that, let $\vv Q$ be a proper subquasivariety of $\vv Q(\mathbb Z)$; then $\mathbb Z \notin \vv Q$. If $\vv Q$ is nontrivial then it contains finitely generated nontrivial algebras and each of them is a direct product of cyclic groups; since any factor of a direct product is also a subalgebra $\vv Q$  must contain at least a  finite cyclic group.   But  any  finite cyclic group does not satisfy $\e_n$ for some (really infinitely many) $n$;  thus $\vv Q$ must be trivial and $\QQ(\mathbb Z)$ is an atom in the lattices of subquasivarieties. Therefore it is  primitive.

Finally there are no other primitive subquasivarieties of the variety of abelian groups, other the ones we described. First observe that the only locally finite primitive subquasivarieties are the locally finite subvarieties.
If $\vv Q$ is a primitive proper subquasivariety of abelian groups that is not locally finite, then $\mathbb Z \in \vv Q$ so $\QQ(\mathbb Z) \sse \vv Q$; since $\vv Q$ is primitive it is structurally complete and thus every proper subquasivariety of $\vv Q$ must generate a proper subvariety of abelian groups (Theorem \ref{structural}) .  But $\QQ(\mathbb Z)$ generates the entire variety of abelian groups so it cannot be proper; hence $\QQ(\mathbb Z) = \vv Q$.

Now an abelian group is a module over $\mathbb Z$ seen as a principal ideal domain and not surprising most of what we said has been generalized to modules over a PID \cite{Belkin1995} and (at some cost) even to modules
over a Dedekind domain \cite{Jedlicka2019}. In general, if $\vv M_\alg R$ is the variety of modules over a Dedekind  domain $\alg R$,
 the primitive subquasivarieties of $\vv M_\alg R$ are exactly the locally finite varieties  and $\QQ(\alg R)$ (regarded as an $\alg R$-module over itself).\qed
\end{example}

\section{More on structural completeness}\label{sec:more}

\subsection{$C$-completeness}\label{Ccompleteness}
The concept of $C$-completeness has been introduced by the second author in \cite{Citkin2016}.
Let $A$ be any set; a {\bf clone} of operations on $A$ is a set of operations on $A$ that contains all the projections and it is closed under composition (whenever possible). As the intersection of any family of clones is still a clone, it makes sense to talk about clone generation (that is of course a closure operator). If $\alg A$ is any algebra, then the {\bf term clone} of $\alg A$, denoted by $\op{Clo}(\alg A)$ is the clone on $A$ generated by all the fundamental operations.

Let now $\vv Q$ be a quasivariety; then the terms in the language of $\vv Q$ can be seen as operations on $\alg F_\vv Q(\o)$, and the set of all terms is just the clone of all derived operations on $\alg F_\vv Q(\o)$, i.e. the clone on $\alg F_\vv Q(\o)$ generated by all the fundamental operations. We will refer to it as the {\bf term clone} of $\vv Q$ and we will denote it by $\op{Clo}(\vv Q)$. Let $C$ be a subclone of  $\op{Clo}(\vv Q)$; a {\bf $C$-quasiequation} is a quasiequation containing only operations from $C$. We say that $\vv Q$ is $C$-structurally complete if for every $C$-quasiequation $\Phi$,
if $\alg F_\vv Q(\o) \vDash \Phi$, then $\vv Q \vDash \Phi$. A quasivariety is $C$-primitive if all its subquasivarieties are $C$-structurally complete.
Observe that if $C'$ is a subclone of $C$ and $\vv Q$ is $C$-structurally complete ($C$ primitive), then $\vv Q$ is $C'$-structurally complete ($C'$-primitive).
Observe also that if $T$ is a set of generators for $C$, it is easy to check that $\vv Q$ is $C$-structurally complete if and only if  for every quasiequation $\Phi$ containing only operations from $T$,
$\alg F_\vv Q(\o) \vDash \Phi$ entails $\vv Q \vDash \Phi$. Therefore if $T$ is a set of terms that generates $C$ we may talk about $T$-structural completeness and $T$-primitivity, meaning the corresponding concept for the clone generated by $T$. If $C$ is the entire term clone of $\vv Q$ then $C$-structural completeness is the usual structural completeness.

 Let $\vv Q$ be a quasivariety and let $C$ a subclone of the term clone of $\vv Q$. Let $\vv Q^C$ be  the class of all $C$-subreducts of algebras in $\vv Q$; then it is easily seen that $\vv Q^C$ is a quasivariety in which all the $C$-quasiequation holding in $\vv Q$ are valid. So if $\vv Q^C$ is structurally complete or primitive, then $\vv Q$ is $C$-structurally complete or $C$-primitive. For instance if $\vv H$ is the variety of Heyting algebras, then its $\{\meet,\imp\}$-subreducts form the variety of Brouwerian semilattices, that is (as we have already observed) primitive; thus $\vv H$ is $\{\meet,\imp\}$-primitive.

  The converse however fails to hold; the variety $\vv H$ of Heyting algebras is $\{\imp,\neg\}$-structurally complete \cite{Mints1971} but the quasivariety of its $\{\imp,\neg\}$-subreducts is not structurally complete \cite{CintulaMetcalfe2010}. The problem is that there is a $\{\imp,\neg\}$-quasiequation that is valid in the in $\alg F_{\vv H^{\{\imp,\neg\}}}(\o)$ but it is not valid in $\alg F_\vv H(\o)$.

\subsection{ u-presentability and Prucnal terms}\label{prucnalterms}

Let $\alg A$ be any algebra, $\th \in \op{Con}(\alg A)$  and let $C$ be a subclone of $\op{Clo}(\alg A)$; by $\alg A^C$ we denote the algebra whose universe is $A$ and whose fundamental operations are those in $C$. We say that $\th$ is {\bf u-presentable} relative to $C$  if there is a set $\Delta \sse \op{Con}(\alg A)$ such that
\begin{enumerate}
\item $\th = \bigcup \Delta$;
\item $\Delta$ is closed under finite joins;
\item $\alg A^C/\d \in \II\SU\PP_u(\alg A^C)$ for all $\d \in \Delta$.
\end{enumerate}
In this case $\Delta$ is called a {\bf u-presentation} of $\th$ relative to $C$.

\begin{theorem} \label{upresentable} Let $\alg A$ be an algebra,  $\th \in \Con A$ and $C$ a subclone of $\op{Clo}(\alg A)$; then
the following are equivalent:
\begin{enumerate}
\item $\th$ is u-presentable relative to $C$;
\item $\alg A^C/\th \in \II\SU\PP_u(\alg A^C)$.
\end{enumerate}
\end{theorem}
\begin{proof} Obviously (2) implies (1).  Let then $\vv U = \II\SU\PP_u(\alg A^C)$ and let $\Delta$ be a u-presentation of $\th$ relative to $C$. For any $\d \in \Delta_\th$, let $L_\d \{\e \in \Delta: \d \le \e\}$.  By the definition of u-presentability  if $\d,\d'\in \Delta$, then $\d \join \d' \in \Delta$, hence
$\{L_\d: \d \in \Delta\}$ has the finite intersection property. Therefore there exists an ultrafilter $U$ on $\Delta$ such that $L_\d \in U$ for all $\d \in \Delta$.

 Consider the map
$$
f: a \longmapsto  (a/\d)_{\d \in \Delta}/U;
$$
this is clearly a homomorphism from $\alg A^C$ to $\prod_{d \in \Delta} \alg A^C/\d$. We claim that $\op{ker}(f) =\th$. In fact if $(a/\d)_{\d \in \Delta}/U =(b/\d)_{\d \in \Delta}/U$, then $\{\d: a/\d = \b /\d\} \in U$.
If $(a,b) \notin \th$ then, as $\th = \bigcup \Delta$, $(a,b) \notin \d$ for all $\d \in J$; this implies that  $\{\d: a/\d = \b /\d\} = \emptyset$ which contradicts the fact that $U$ is an ultrafilter.
Hence $(a,b) \in \th$ and $\th = \op{ker}(f)$.

Conversely assume that $(a,b) \in \th$; then $(a,b) \in \d'$ for some $\d'\in \Delta$. It follows that $(a,b) \in \e$ for all $\e \in L_{\d'}$ and thus
$$
L_{\d'} \sse \{\d: a/\d = b/\d\}.
$$
As $U$ extends  $\{L_\d: \d \in J\}$,  the right hand side of the above inclusion belongs to $U$, which means that $(a/\d)_{\d \in J}/U =(b/\d)_{\d \in J}/U$ as wished.
Therefore $\alg A^C/\d \in \II\SU\PP_u(\alg A^C)$.
\end{proof}

\begin{corollary} For any algebra $\alg A$ and $C \sse \op{Clo}(\alg A)$ the following are equivalent:
\begin{enumerate}
\item every congruence of $\alg A$ is u-presentable relative to $C$;
\item every compact congruence of $\alg A$ is $u$-presentable relative to $C$.
\end{enumerate}
\end{corollary}
\begin{proof} (1) trivially implies (2). Let $\th \in \Con A$ and let $\Delta_\th = \{\d \in \Con A: \d \le \th,\ \text{$\d$ compact}\}$.
Then $\Delta_\th$ is closed under finite joins an $\bigcup\Delta_\th =\th$. Moreover by Theorem \ref{upresentable}, $\alg A^C/\d \in \II\SU\PP_u(\alg A^C)$ for all
$\d \in \Delta_\th$. Therefore $\th$ is u-presentable relative to $C$.
\end{proof}

Now we can connect u-presentability and structural completeness, but first we need a lemma.

\begin{lemma} Let $\vv Q$ be a quasivariety and let
$$
\Phi := \bigwedge_{i=1}^n r_i(\mathbf x) \app s_i(\mathbf x) \Rightarrow r(\mathbf x) \app s(\mathbf x)
$$
be a quasiequation. Then $\vv Q \vDash\Phi$ if and only if $\alg A \vDash \Phi$ where $\alg A$ is an algebra in $\vv Q$ finitely presented by the relations
$r_i(\mathbf x) \app s_i(\mathbf x)$, $i=1,\dots,n$.
\end{lemma}
\begin{proof} The left-to-right implication is trivial. Suppose that $\alg A \vDash \Phi$ and assume by way of contradiction
that there is $\alg B \in\vv  Q$ with $\alg B \not\vDash\Phi$.
Then there are $\vuc bn \in B$ such that
$$
r_i(\vuc bn) = s_i(\vuc bn)\quad\text{ for all $i$}
$$
 but $r(\vuc bn) \ne s(\vuc bn)$.

As $\alg A$ is finitely presented and generated by $\vuc an$ we can extend the
map $f : a_i \longmapsto b_i$ to a homomorphism $f : \alg A \longrightarrow \alg B$. By assumption, $\alg A \vDash \Phi$ and thus $r(\vuc an) = s(\vuc an)$.
Therefore
\begin{align*}
r(\vuc bn) &= r (f(a_1),\dots,f(a_n)) \\
&= f(r(\vuc an)) = f(s(\vuc an)) \\
&= s(f(a_1),\dots,f(a_n)) = s(\vuc bn).
\end{align*}
 So $\alg B\vDash \Phi$, a clear contradiction.
 \end{proof}

\begin{corollary}\label{compact}  Let $\vv Q$ be a quasivariety and $\Phi$ be a quasiequation. Then $\vv Q \vDash \Phi$ if and only if for every compact $\vv Q$-congruence
$\th$ of $\alg F_\vv Q(\o)$ we have $\alg F_\vv Q(\o)/\th \vDash \Phi$.
\end{corollary}

\begin{theorem} Let $\vv Q$ be a quasivariety, $\alg A\in \vv Q$ and $C$ a clone of operations of $\vv Q$; if every compact congruence of $\alg F_\vv Q(\o)$ is u-presentable relative to $C$, then $\vv Q$ is $C$-structurally complete.
\end{theorem}
\begin{proof} Let $\Phi$ be a $C$-quasiequation and suppose that $\vv Q \not\vDash \Phi$; then by Corollary \ref{compact} there exists a compact $\vv Q$-congruence $\th$ of $\alg F_\vv Q (\o)$ such that $\alg F_\vv Q(\o)/\th \not\vDash \Phi$.  Since $\th$ is u-presentable relative to $C$, $\alg F^C_\vv Q/\th$ embeds in an ultrapower of $\alg F^C_\vv Q(\o)$ and so $\alg F^C_\vv Q(\o) \not\vDash \Phi$.  As $\Phi$ is a $C$-quasiequation $\alg F_\vv Q(\o)  \not\vDash \Phi$,
hence the thesis holds.
\end{proof}

Given a quasivariety $\vv Q$ and a countably generated algebra $\alg A \in \vv Q$, it is well-known that there is an epimorphism $\f$ from $\op{Clo}(\alg F_\vv Q(\o))$ to $\op{Clo}(\alg A)$ (regarded as algebras in the type of $\vv Q$). If $C \sse \op{Clo}(\alg F_\vv Q(\o))$ we say that $\alg A$ is u-presentable relative
to $C$ if it is u-presentable relative to $\f(C)$.

\begin{corollary}  Let $\vv Q$ be a quasivariety, $\alg A\in \vv Q$ and $C$ a clone of operations of $\vv Q$; if every compact  $\vv Q$-congruence of every countably generated algebra in $\vv Q$ is u-presentable with respect to $C$, then $\vv Q$ is $C$-primitive.
\end{corollary}
\begin{proof} It is enough to observe that if $\vv Q'\sse \vv Q$, then $\alg F_{\vv Q'}(\o)$ is a countably generated algebra in $\vv Q$.
\end{proof}

It is interesting to observe that if $C$ is the entire term clone of $\vv Q$, then we obtain a new necessary and sufficient condition.

\begin{theorem}\label{thm: cmi and u-presentable} A quasivariety $\vv Q$ is structurally complete if and only if every completely meet irreducible congruence $\th \in \op{Con}_\vv Q(\alg F_\vv Q(\o))$ is u-presentable.
\end{theorem}
\begin{proof} If $\vv Q$ is structurally complete, then $\vv Q = \QQ(\alg F_\vv Q(\o))$ and therefore every $\vv Q$-irreducible algebra is in $\II\SU\PP_u(\alg F_\vv Q(\o))$. It follows that every completely meet irreducible congruence of $\alg F_\vv Q(\o)$ is  u-presentable.

Conversely if (2) holds then, by  Theorem \ref{upresentable}, if $\th$ is a  completely meet irreducible congruence of $\alg F_\vv Q(\o)$ then $\alg F_\vv Q (\o)/\th \in \II\SU\PP_u(\alg F_\vv Q(\o))$. Every $\vv Q$-irreducible  algebra $\alg A$ is embeddable in an ultraproduct of finitely generated subalgebras and each one of them is a subdirect product of finitely generated $\vv Q$-irreducible algebras that are necessarily equal to $\alg F_\vv Q(\o)/\d$ for some completely meet irreducible congruence
of $\alg F_\vv Q(\o)$. Then each one of them is in $\II\SU\PP_u(\alg F_\vv Q(\o))$ and so does $\alg A$, since it is $\vv Q$-irreducible.   It follows that $\vv Q = \QQ(\alg F_\vv Q(\o))$ and hence $\vv Q$ is structurally complete.
\end{proof}

\begin{corollary} \label{cor: cmi and u-presentable} A quasivariety $\vv Q$ is primitive if and only if for every countably generated $\alg A \in \vv Q$ every completely meet irreducible $\th \in \op{Con}_\vv Q(\alg A)$ is u-presentable.
\end{corollary}

Sometimes $u$-presentability is expressible directly via term operations.
Let $\alg A$ be any algebra, $C$ a subclone of the clone of all term operations on $\alg A$, $T$ a set of generators for $C$ and $\alg A^T$ the reduct of
$\alg A$ to $T$; we say that $\alg A$ has the {\bf Prucnal property} relative to $C$ if for all $n \in \mathbb N$ there is a term $t_n(\vuc xn,\vuc yn,z)$ such that for any compact $\th \in \op{Con}_\vv Q(\alg A)$, $\th = \bigvee_{i=1}^n  \cg^\vv Q_\alg A(a_i,b_i)$
\begin{enumerate}
\item   the map $\sigma_n : c \longmapsto t_n(\vuc an,\vuc bn, c)$ is an endomorphism of $\alg A^T$;
\item   $\op{ker}(\sigma_n) = \th$.
\end{enumerate}
In this case the terms $\sigma_n$ are called the {\bf Prucnal terms} relative to $C$ and the endomorphisms $\sigma_n$ are called the {\bf Prucnal $C$-endomorphisms}.
If $C$ is the entire clone of derived operations on $\alg A$, then we will drop the decoration $C$.

\begin{theorem} Let $\vv Q$ be a quasivariety and let $C$ a clone of term operations of $\vv Q$. If $\alg F_\vv Q(\o)$ has the Prucnal property relative to $C$, then $\vv Q$ is $C$-structurally complete.
\end{theorem}
\begin{proof} By Corollary \ref{compact} it is enough to show that every compact congruence of $\alg F_\vv Q(\o)$ is u-presentable relative to $C$. So let $\th$ be such a congruence; then from the definition of Prucnal terms and  a straightforward application of the Third Homomorphism Theorem it follows that
$\alg F^C_\vv Q(\o)/\th \in \II\SU(\alg F^C_\vv Q(\o)$. This implies that $\th$ is $u$-presentable.
\end{proof}

\begin{corollary}\label{cor:compact} Let $\vv Q$ be a quasivariety and let $C$ a clone of term operations of $\vv Q$. If every countably generated algebra in $\vv Q$  has the Prucnal property relative to $C$, then $\vv Q$ is $C$-primitive.
\end{corollary}

A quasivariety $\vv Q$ has the {\bf principal Prucnal property} relative to $C$, if there is a term $t(x,y,z)$ that is a Prucnal term for principal congruences, relative to $C$, i.e. for all $\alg A \in \vv Q$ and for all $a,b, \in \alg A$
\begin{enumerate}
\item the map $\sigma : c \longmapsto t(a,b,c)$ is an endomorphism of $\alg A^C$;
\item $\op{ker}(\sigma) = \cg_\alg A^\vv Q(a,b)$.
\end{enumerate}
We will show that for any quasivariety the principal Prucnal property (relative to $C$) implies the Prucnal property (relative to $C$).
To this aim we need several lemmas.

\begin{lemma} \label{techlemma1} \cite{CzelakowskiDziobiak1990} Let $\vv Q$ be a quasivariety and $\alg A \in \vv Q$; if $\th \in \op{Con}_\vv Q(\alg A)$ and $a,b \in A$ then
$(\th \join \cg_\alg A^\vv Q(a,b))/\th = \cg_{\alg A/\th}(a/\th,b/\th)$.
\end{lemma}

For the following lemma, in order to avoid clutter we use a special notation; first we will denote a sequence $\vuc an \in A$ by
$\mathbf a^n$. Next if $\th \in \Con A$ we will write $\bar {\alg A}$ for $\alg A/\th$, $\ol a$ for $a/\th$ and
$\ol{{\mathbf a}}^n$for $\vuc {\ol{a}}n$.

\begin{lemma}\label{techlemma2} Let $\vv Q$ be a quasivariety, $\alg A \in \vv Q$ and $\vuc an,\vuc bn \in A$. If $\overline x = x/\cg_\alg A^\vv Q(a_n,b_n)$, $\bar {\alg A} = \alg A/\cg_\alg A^\vv Q(a_n,b_n)$ and
$c,d \in A$ then
\begin{align*}
&(c,d) \in \cg_\alg A^\vv Q(a_1,b_1) \join \dots \join \cg_\alg A^\vv Q(a_n,b_n)\quad\text{if and only if}\\
&(\overline  c,\overline  d) \in \cg_{\bar{\alg A}}^\vv Q({\overline  a}_1,{\overline  b}_1) \join \dots \join \cg_{\bar{\alg A}}^\vv Q({\overline  a}_{n-1},{\overline  b}_{n-1}).
\end{align*}
\end{lemma}
\begin{proof} The proof is by induction, using Lemma \ref{techlemma1} both in the base step and the induction step.
\end{proof}

Let $\vv Q$ be a quasivariety with a principal Prucnal term relative to $C$, say $t(x,y,z)$.
We define for $n \ge 1$
\begin{align*}
&t_1(x_1,y_1,z) := t(x_1,y_1,z)\\
&t_{n+1} = t_n(\vuc xn,\vuc yn, t(x_n,y_n,z)).
\end{align*}

\begin{lemma}\label{sigmakernel} Let $\vv Q$ be a quasivariety with principal Prucnal term $t(x,y,z)$ relative to $C$, $\alg A \in \vv Q$ and $a,b,c,d \in A$.
Then
\begin{align*}
&(c,d) \in \cg_\alg A^\vv Q(a_1,b_1) \join\dots\join \cg_\alg A^\vv Q (a_n,b_n)\quad\text{if and only if}\\
&t_n({\mathbf a}^n,{\mathbf b}^n,c) = t_n({\mathbf a}^n,{\mathbf b}^n,d).
\end{align*}
\end{lemma}
\begin{proof} We induct on $n\ge 1$; the case $n=1$ comes straight from the definition of principal Prucnal term.
Suppose then that the conclusion holds for any positive integer less than $n$. Then
\begin{align*}
&(c,d) \in \cg_\alg A^\vv Q(a_1,b_1) \join\dots\join \cg_\alg A^\vv Q (a_n,b_n)\quad\text{if and only if}\\
&(\overline  c,\overline  d) \in \cg_{\bar{\alg A}}^\vv Q({\overline  a}_1,{\overline  b}_1) \join \dots \join \cg_{\bar{\alg A}}^\vv Q({\overline  a}_{n-1},{\overline  b}_{n-1})\quad\text{if and only if}\\
&t_{n-1}(\ol{{\mathbf a}}^{n-1},\ol{{\mathbf b}}{}^{n-1}, \overline c) = t_{n-1}(\ol{{\mathbf a}}^{n-1},\ol{{\mathbf b}}{}^{n-1}, \overline d)\quad\text{if and only if}\\
 &t_{n-1}(\ol{a}^{n-1},\ol{b}^{n-1}, \ol{d})
d\quad \mathrel{\cg_\alg A^\vv Q(a_n,b_n)}\quad t_{n-1}(\ol{a}^{n-1},\ol{b}^{n-1}, d) \quad\text{if and only if}\\
&t(a_n,b_n,t_{n-1}(\ol{a}^{n-1},\ol{b}^{n-1}, c))  = t(a_n,b_n,t_{n-1}(\ol{a}^{n-1},\ol{b}^{n-1}, d)) \quad\text{if and only if}\\
&t_n({\mathbf a}^n,{\mathbf b}^n,c) = t_n({\mathbf a}^n,{\mathbf b}^n,d).
\end{align*}
\end{proof}

We observe that Lemma \ref{sigmakernel} is a generalization of Theorem 2.6 in \cite{EDPC3}.

\begin{corollary} If a quasivariety $\vv Q$ has a principal Prucnal term relative to $C$, then it has the Prucnal property relative to $C$.
\end{corollary}
\begin{proof} The terms $t_n, n \ge 1$ clearly satisfy the first condition for Prucnal terms, as iterated compositions of $t$. By Lemma \ref{sigmakernel} they also satisfy the second.
\end{proof}

\subsection{The relative TD-term}

Here we will consider a special principal Prucnal term.
Let $\vv Q$ be a quasivariety; a {\bf relative TD-term} for $\vv Q$ is a ternary term $t(x,y,z)$ such that
\begin{enumerate}
\ib $\vv Q \vDash t(x,x,z)$;
\ib  if $(c,d) \in \cg_\alg A^\vv Q(a,b)$ then  $t(a,b,c) = t(a,b,d)$.
\end{enumerate}

Let $\vv Q$ be a quasivariety with a relative TD-term $t(x,y,z)$ and $q$ a $k$-term of $\vv Q$; we say that {\bf $q$ commutes with $t$} if for all $\alg A\in \vv Q$ and $a,b,\vuc ck \in A$
$$
t(a,b,q(\vuc ck)) = q(t(a,b,c_1),\dots,t(a,b,c_k)).
$$

\begin{theorem} If $t(x,y,z)$ is a relative TD-term for $\vv Q$, then $t$ is a principal Prucnal term relative to any clone $C$ of operations that commute with
$t(x,y,z)$.
\end{theorem}
\begin{proof} Let $\alg A \in \vv Q$ and let $a,b \in A$ and consider the mapping
$$
z \longmapsto t(a,b,z).
$$
If $q \in C$ and $\vuc ck \in A$, then
$$
t(a,b,q(\vuc ck)) = q(t(a,b,c_1),\dots,t(a,b,c_k)),
$$
so the mapping is an endomorphism of $\alg A^C$.

Now if $(c,d) \in \cg_\alg A^\vv Q(a,b)$ then $t(a,b,c) = t(a,b,d)$ by definition; conversely if $t(a,b,c) =t(a,b,d)$ then
$$
c = t(a,a,c) \mathrel{\cg_\alg A^\vv Q(a,b)} t(a,b,c) = t(a,b,d)\mathrel{\cg_\alg A^\vv Q(a,b)} t(a,a,d) =d.
$$
\end{proof}

\begin{corollary}\label{prucnalTD} Let $\vv Q$ be a quasivariety with a relative TD-term $t(x,y,z)$. Then for all nontrivial $\alg A \in \vv Q$, $\alg A$ has the Prucnal property for $\alg A$, relative to any clone $C$ of operations that commute with the relative TD-term.
\end{corollary}

\begin{corollary}  Let $\vv Q$ be a quasivariety with a relative TD-term; then $\vv Q$ is $C$-primitive for any clone $C$ of terms that commute with the relative TD-term.
\end{corollary}

The notion of relative TD-term generalizes to quasivarieties the notion of TD-term in \cite{EDPC3} and the results in this section
generalizes to quasivarieties the results in \cite{Citkin2016}. Some of the others results in \cite{EDPC3} can be generalized to quasivarieties with a relative TD-term.

A quasivariety has {\bf definable principal relative congruences} (DPRC) if there is a first order formula $\Phi(x,y,z,w)$ in the language of of $\vv Q$ such that for all $\alg A$ in $\vv Q$ and for all $a,b,c,d \in \alg A$
$$
(c,d) \in \cg_\alg A^\vv Q(a,b) \quad\text{if and only if} \quad \alg A\vDash \Phi(a,b,c,d).
$$
A quasivariety has {\bf equationally definable principal relative congruences} (EDPRC) if $\Phi$ can be taken as a finite conjunction of equations.
A quasivariety $\vv Q$ has the {\bf relative congruence extension property} (RCEP) if for any $\alg A \in \vv Q$ and any subalgebra $\alg B$ of $\alg A$, if $\f \in \op{Con}_\vv Q(\alg B)$ there is a $\th \in \op{Con}_\vv Q(\alg A)$ with $\f = \th \cap B^2$.

\begin{theorem} If $\vv Q$ has a relative TD-term $t(x,y,z)$ then
\begin{enumerate}
\item   for all $\alg A \in \vv Q$ the lattice $\op{Con}_\vv Q(\alg A)$ is dually relatively pseudocomplemented;
\item   $\vv Q$ has EDPRC: for any $\alg A \in \vv Q$ and $a,b,c,d \in A$, $(c,d) \in \cg_\alg A^\vv Q(a,b)$ if and only if $t(a,b,c) = t(a,b,d)$;
\item   $\vv Q$ has the RCEP ;
\item    the $\vv Q$-congruences of any $\alg A \in \vv Q$  3-permute.
\end{enumerate}
\end{theorem}

The proofs are easily patterned after the corresponding proofs in \cite{EDPC3}.

\section{Other primitive quasivarieties}\label{sec:idiscriminator}

\subsection{Congruence intersection terms}

There are quasivarieties not necessarily having Prucnal terms to which some of the techniques in the previous sections can be applied (at a cost).
We will say that a quasivariety has the {\bf relative principal intersection property}  (RPIP) if there are quaternary terms $p,q$ in the type of $\vv Q$ such that for any $\alg A \in \vv Q$
$$
\cg^\vv Q_\alg A(a,b) \cap \cg^\vv Q(c,d) = \cg_\alg A^\vv Q(p(a,b,c,d),q(a,b,c,d)).
$$
The RPIP for varieties was introduced by K. Baker in \cite{Baker1974} and studied extensively in \cite{AglianoBaker1999a}.   We observe that if $\op{Con}_\vv Q(\alg A)$ is distributive, then (as any complete distributive lattice) it is {\bf pseudocomplemented}: in particular if $\alg A \in \vv Q$, then for any $\th \in \op{Con}_\vv Q(\alg A)$ there is a $\th^* \in \op{Con}_\vv Q(\alg A)$ such that
$$
\th^* = \bigvee\{\f \in \op{Con}_\vv Q(\alg A): \th \meet \f = 0_\alg A\}.
$$
In particular, if $\op{MI}^\vv Q_\alg A$ is the set of  meet irreducible elements in $\op{Con}_\vv Q(\alg A)$ and $a,b \in A$, then
$$
(\cg_\alg A^\vv Q(a,b))^* = \bigwedge\{\f \in \op{MI}_\alg A^\vv Q: (a,b) \notin \f\}.
$$
From now on we will denote by $\g_\alg A^\vv Q(a,b)$ the pseudocomplement of $\cg_\alg A^\vv Q(a,b)$ in $\op{Con}_\vv Q(\alg A)$.

\begin{theorem}\label{RPIP} For a quasivariety $\vv Q$ the following are equivalent:
\begin{enumerate}
 \item $\vv Q$ has the RPIP;
\item $\vv Q$ is relative congruence distributive  and there are terms $p,q$ such that for all $\vv Q$-irreducible $\alg A \in \vv Q$ and $a,b,c,d \in A$
$$
p(a,b,c,d) = q(a,b,c,d) \quad\text{if and only if}\quad\text{$a=b$ or $c=d$}.
$$
\item $\vv Q$ is relative congruence distributive and there are terms $p,q$ such that for all $\alg A \in \vv Q$ and $a,b \in A$
$$
\g_\alg A^\vv Q(a,b) = \{(c,d): p(a,b,c,d) = q(a,b,c,d)\}.
$$
\end{enumerate}
\end{theorem}
\begin{proof} Observe that any quasivariety with the RPIP is relatively congruence distributive, by Proposition 1.2 in \cite{CzelakowskiDziobiak1990}.
Let's  assume (1), i.e. $\vv Q$ has the RPIP witnessed by $p,q$; let $\alg A\in \vv Q$, $a,b \in A$ and
$$
\a = \{(c,d): p(a,b,c,d)=q(a,b,c,d)\}.
$$
Then we have
\begin{align*}
&(c,d) \in \a \quad\text{if and only if}\\
&p(a,b,c,d) =q(a,b,c,d)  \quad\text{if and only if}\\
&\cg_\alg A\vv Q(a,b) \cap \cg_\alg A^\vv Q(c,d) = 0_\alg A \quad\text{if and only if}\\
&\cg_\alg A^\vv Q(c,d) \le \g_\alg A^\vv Q(a,b) \quad\text{if and only if}\\
&(c,d) \in \g_\alg A^\vv Q(a,b)
\end{align*}
so $\a =\g_\alg A^\vv Q(a,b)$ and (1) implies (3).

Next assume (3) and let $\alg A$ be a $\vv Q$-irreducible algebra in $\vv Q$. Then
$$
\g_\alg A^\vv Q(a,b) = \left\{
                         \begin{array}{ll}
                           1_\alg A, & \hbox{if $a=b$;} \\
                           0_\alg A, & \hbox{if $a \ne b$.}
                         \end{array}
                       \right.
$$
Suppose that $p(a,b,c,d) = q(a,b,c,d)$ with $a\ne b$; then  $(c,d) \in \g_\alg A^\vv Q(a,b)$ and thus  $c=d$.
Conversely, if  $a=b$ then $\g_\alg A^\vv Q(a,b) =1_\alg A$ so  $p(a,b,c,d) =q(a,b,c,d)$ for all $c,d$; if $c=d$ then $(c,d) \in \g_\alg A^\vv Q(a,b)$
and so again $p(a,b,c,d) = q(a,b,c,d)$. Therefore  (3) implies (2).

Finally assume  (2). We will make use of the following  fact whose proof can be found in \cite{CzelakowskiDziobiak1990}: for any $\th \in \op{Con}_\vv Q(\alg A)$ and $a,b \in A$, $(\th \join \th^\vv Q_\alg A(a,b))/\th = \cg^\vv Q_{\alg A/\th}(a/\th,b/\th)$. Let then $a,b,c,d \in A$; we must prove that
$$
\cg^\vv Q_\alg A(a,b) \cap \cg^\vv Q(c,d) = \cg_\alg A^\vv Q(p(a,b,c,d),q(a,b,c,d)).
$$
Since $\op{Con}_\vv Q(\alg A)$ is an algebraic lattice, every element is a meet of completely meet irreducible elements. So to prove (1), it is enough to show that, if $\f$ is completely meet irreducible in $\op{Con}_\vv Q(\alg A)$ then
$$
\cg^\vv Q_\alg A(a,b) \cap \cg^\vv Q(c,d) \le \f\quad\text{if and only if}\quad \cg_\alg A^\vv Q(p(a,b,c,d),q(a,b,c,d)) \le \f.
$$
So suppose that $\cg^\vv Q_\alg A(a,b) \cap \cg^\vv Q(c,d) \le \f$; as $\f$ is completely meet irreducible and $\op{Con}_\vv Q(\alg A)$ is distributive, $\f$ is join prime, so either $\cg^\vv Q_\alg A(a,b) \le \f$ or $ \cg^\vv Q(c,d) \le \f$. Then in $\alg A/\f$ we have
\begin{align*}
\cg^\vv Q_{\alg A/\f}(a/\th,b/\f) &\cap \cg^\vv Q_{\alg A/\th}(c/\f,d/\f) \\
                             &= (\f \join \th^\vv Q_\alg A(a,b))/\f \cap (\f \join \th^\vv Q_\alg A(a,b))/\f\\
&= \f/\f = 0_{\alg A/\f}.
\end{align*}
Since $\alg A/\f$ is $\vv Q$-irreducible $p(a,b,c,d)/\f = q(a,b,c,d)/\f$ which implies
$$
\cg_\alg A^\vv Q(p(a,b,c,d),q(a,b,c,d)) \le \f.
$$
Conversely if  $\cg_\alg A^\vv Q(p(a,b,c,d),q(a,b,c,d)) \le \f$, then
\begin{align*}
\cg_{\alg A/\f}^\vv Q(p(a,b,c,d)/&\f, q(a,b,c,d)/\f)\\
& = (\f \join \cg_\alg A^\vv Q(p(a,b,c,d),q(a,b,c,d)))/\f = \f/\f = 0_{\alg A/\f}
\end{align*}
so, since $\alg A/\f$ is $\vv Q$-irreducible, either $(a,b) \in \f$ or $(c,d) \in \f$, which implies $\cg_\alg A^\vv Q(a,b) \cap \cg_\alg A^\vv Q(c,d) \le \f$ as wished.
So (3) implies (1) and the proof is finished.
\end{proof}

Let  $\vv Q$ be a quasivariety with the  RPIP (and hence relative congruence distributive), $\alg A\in \vv Q$ and let $\alg L$ be the lattice of compact $\vv Q$-congruences of $\alg A$ plus eventually $1_\alg A$. Then  one can define a topology on $\op{MI}_\alg A^\vv Q$ which is equivalent to the Priestley duality on the lattice of ideals $\alg L$ (see \cite{CampercholiVaggione2012}, Proposition 3.10).  More precisely the topology
is defined by the subbase $\op{MI}_\alg A^\vv Q(a,b)= \{\th \in \op{MI}_\alg A: (a,b) \notin \th\}$, for $a,b \in A$.  It follows that $\op{MI}_\alg A^\vv Q$ is compact in that topology which in turn implies (see Lemma 2.1 in \cite{Vaggione1995}):

\begin{lemma}\label{lemma: vaggione} Let $\vv Q$ be a quasivariety with the RPIP and $\alg A \in \vv Q$; then if $\bigwedge \op{MI}_\alg A^\vv Q(a,b) \le \g$, where $\g \in \op{MI}^\vv Q_\alg A$, then there
exists a $\th \in \op{MI}_\alg A^\vv Q(a,b)$ with $\th \le \g$.
\end{lemma}

\begin{theorem} \label{thm: join coprincipal} Let $\vv Q$ be a quasivariety with the RPIP such that
 every finitely $\vv Q$-irreducible  algebra in $\vv Q$ is $\vv Q$-simple. Then the following hold.
\begin{enumerate}
\item For any  $\alg A \in \vv Q$ and $a,b \in A$, $\cg_\alg A^\vv Q(a,b)$ and $\g_\alg A^\vv Q(a,b)$ are complements; therefore
the meet of two coprincipal  $\vv Q$-congruences of $\alg A$ is coprincipal.
\item If $\th \in \op{Con}_\vv Q(\alg A)$ is such that $\alg A/\th$ is $\vv Q$-simple, then
$$
\bigcup\{\gamma_\alg A^\vv Q(a,b):  \gamma_\alg A^\vv Q(a,b)\le \th\} = \th.
$$
\end{enumerate}
\end{theorem}
\begin{proof} By hypothesis every element of $\op{MI}_\alg A^\vv Q$ is a maximal element in in $\op{Con}_\vv Q(\alg A)$. Let $\g \ge \gamma_\alg A^\vv Q(a,b)$ with $\g \in \op{MI}_\alg A^\vv Q$; then
$$
\bigwedge \op{MI}_\alg A^\vv Q(a,b) = \g_\alg A^\vv Q(a,b) \le \g
$$
so by Lemma \ref{lemma: vaggione} there is a nontop $\th \in \op{MI}_\alg A^\vv Q(a,b)$ with $\th\le \g$. By maximality of $\th$ we must have $\th = \g$ and thus $(a,b) \notin \g$.

Now suppose that $\cg_\alg A^\vv Q(a,b) \join \g_\alg A^\vv Q(a,b) \ne 1_\alg A$; then there  must be a maximal congruence $\g$ with $\cg_\alg A^\vv Q(a,b) \join \g_\alg A^\vv Q(a,b) \le \g$; as $\g \in \op{MI}_\alg A^\vv Q$ and
$\g \ge \g_\alg A(a,b)$ we must have $(a,b) \notin \g$, a clear contradiction. So $\cg_\alg A^\vv Q(a,b) \join \g_\alg A^\vv Q(a,b) = 1_\alg A$ and the first part of the thesis follows.

That the meet of two coprincipal  $\vv Q$-congruences of $\alg A$ is coprincipal comes straight from the properties of complemented elements in a distributive lattice and this concludes the proof of (1).

For (2) we will show that for all $\cg_\alg A^\vv Q(a,b) \le \th$ we have $\gamma_\alg A^\vv Q(a,b) \le \th$. As
$\bigcup\{\cg_\alg A^\vv Q(a,b): \cg_\alg A^\vv Q(a,b) \le\th\} =\th$ the conclusion will follow.

If $\cg_\alg A (a,b) \le \th$ then $\cg_\alg A (a,b) \join \g_\alg A (a,b) = 1_\alg A$ and, since $\th$ is covered by $1_\alg A$ there
exists $(c,d)\in\g_\alg A (a,b)$ with $(c,d) \notin \th$.
Now $\cg_\alg A (c,d) \le \g_\alg A (a,b)$ so $\cg_\alg A (a,b) \le \g_\alg A (c,d)$; moreover
$\cg_\alg A (c,d) \cap \g_\alg A (c,d) = 0_\alg A \le \th$ and, as $\th$ is meet prime (since it is meet irreducible and $\op{Con}_\vv Q A$ is distributive),
we must have $\g\alg A (c,d) \le \th$.
\end{proof}

\subsection{Dual i-discriminator varieties}

In the late 1990's K. Baker and the second author collaborated for some time (\cite{AglianoBaker1999a}, \cite{AglianoBaker1999}; while writing the second paper an interest arose about various kinds of dual discriminator functions and produced an unpublished (and very rough) manuscript \cite{AglianoBaker1999c}, that had a very limited circulation at the time (but some of the results have been rediscovered in \cite{BignallSpinks2019}). In this section we improve on those results and we show how they connect to the problems we are considering.

The {\bf dual i-discriminator function} on a set $A$ is a ternary function $p(x,y,z)$ such that for any $a,b,c \in A$
\begin{align*}
&a\ne b \quad \text{implies} \quad p(a,b,c)= c\\
&p(a,a,b) = p(a,a,c)\\
&p(a,a,a) = p(p(a,a,a),p(a,a,a),p(a,a,a)).
\end{align*}
If we write $\pi(x) = p(x,x,x)$, this definition becomes:
$$
\pi(\pi(a)) = \pi(a)\quad\text{and}\quad p(a,b,c) = \begin{cases}
\pi(a)&\text{if $a=b$}\\
c&\text{if $a \ne b$}\end{cases}.
$$
 In case $\pi(x)$ is the identity, then the dual i-discriminator function is just the dual discriminator function described in \cite{FriedPixley1979}.  We first observe the following easy fact:

\begin{lemma}\label{semi} If $\alg A$ is an algebra and $p(x,y,z)$ is a dual i-discriminator term for $\alg A$, then $\alg A$ is simple.
\end{lemma}
\begin{proof} If $\th \in \Con A$ and there are $(a,b) \in \th$ with $a \ne b$, then for all $c,d \in A$
$$
c = p(a,b,c) \mathrel{\th} p(a,a,c) = p(a,a,d) \mathrel{\th} p(a,b,d) = d.
$$
Therefore $\th = 1_\alg A$ and $\alg A$ is simple.
\end{proof}

A quasivariety $\vv Q$ is a {\bf dual i-discriminator} quasivariety if there is term $p(x,y,z)$ that is the dual i-discriminator on every $\vv Q$-irreducible algebra in $\vv Q$. Observe that, by Lemma \ref{semi}, if $\vv Q$ is a dual i-discriminator variety, then every $\vv Q$-irreducible algebra in $\vv Q$ is simple; so $\vv Q$ is semisimple in the absolute sense.

Dual i-discriminator varieties share many properties with dual discriminator varieties in the sense of \cite{FriedPixley1979}. However
there are dual i-discriminator quasivarieties that are not dual discriminator quasivarieties (see Example \ref{implalg} below).
First we want to show that every dual i-discriminator quasivariety is relative congruence distributive. In order to do that first we show:

\begin{theorem}\label{dualdist} Let $\vv K$ be a class of similar algebras
with a term $p(x,y,z)$ that is the dual i-discriminator on
each member of $\vv K$. Then $\VV(\vv K)$ is congruence distributive and in
fact satisfies J\'onsson's condition
$\Delta_3$.
\end{theorem}
\begin{proof}
We must produce four ternary terms
$t_0,t_1,t_2,t_3$ such that $t_0(x,y,z) \approx z$, $t_3(x,y,z)\approx z$
and moreover
\begin{enumerate}
\item $t_i(x,y,x) \approx x$ for $i=1,3$,
\item $t_1(x,x,z) \approx x$,
\item $t_2(x,x,z) \approx z$,
\item $t_1(x,z,z) \approx t_2(x,z,z)$
\end{enumerate}
all hold in $\VV(\vv K)$. Let
$$
t_1(x,y,z) = p(p(x,y,z),y,x)\qquad\qquad t_2(x,y,z) = p(p(x,y,y),z,z).
$$
Since conditions 1.--4. are equational, it is enough to show that
they hold for each  member of $\vv K$.
Accordingly, let $\alg A \in \vv K$  and let
$a,b,c \in A$. Then $t_1(a,b,a) = p(p(a,b,a),b,a)$. If $a\ne b$, then
$p(p(a,b,a),b,a)= p(a,b,a) = a$.
If $a=b$  then $p(p(a,b,a),b,a) =p(\pi(a),a,a)=a$, whether or not
$\pi(a) = a$. An identical argument shows that $t_2(a,b,a) =a$, thus proving
1.

Next, $t_1(a,a,c) = p(p(a,a,c),a,a) = p(\pi(a),a,a) =a$ as
above. On the other hand $t_2(a,a,c) = p(p(a,a,a),c,c) = p(\pi(a),c,c)$,
which is equal to $c$ if $\pi(a) \ne c$, or to
$\pi\pi(a) = \pi(a) = c$ if $\pi(a)=c$. This proves 2. and 3.

Finally we have to show that $t_1(a,c,c) = t_2(a,c,c)$. If
$a = c$, both sides reduce to $p(p(a,a,a),a,a)$. If
$a\ne c$ we have
\begin{align*}
&t_1(a,c,c) = p(p(a,c,c),c,a) = p(c,c,a) = \pi(c)\\
&t_2(a,c,c) = p(p(a,c,c),c,c) = p(c,c,c)=  \pi(c).
\end{align*}
This proves 4. and therefore $\VV(\vv K)$ is congruence distributive.
\end{proof}

Next:

\begin{lemma}\label{all} Let $\vv K$ be a class of similar algebras with
a term $p(x,y,z)$ that is the dual i-discriminator on
each member of $\vv K$. Then $\vv Q=\QQ(\vv K)$ is a dual i-discriminator
variety with dual i-discriminator term $p(x,y,z)$. Moreover $p(x,y,z)$ is
a dual i-discriminator term for any finitely $\vv Q$-irreducible algebra in $\vv Q$.
\end{lemma}
\begin{proof} Since the second claim clearly implies the first we will prove that one.  By Theorem
\ref{quasivariety}
all the finitely $\vv Q$-irreducible members of $\vv Q$ lie in
$\II\SU\PP_u(\vv K)$. Since being a dual i-discriminator function
can be expressed by the universal sentence
\begin{equation*}
\forall x,y,z[p(x,x,y) \approx p(x,x,z)] \meet [\neg (x\approx y) \imp
p(x,y,z)\app \pi(z)],
\end{equation*}
$p(x,y,z)$ is a dual i-discriminator on each member of $\II\SU\PP_u(\vv
K)$. This concludes the proof.
\end{proof}

\begin{theorem} Every dual i-discriminator quasivariety is relative congruence distributive.
\end{theorem}
\begin{proof} Let $\vv Q$ be a dual i-discriminator variety; then by Lemma \ref{all} $\VV(\vv Q)$ is congruence distributive. Now in \cite{Dziobiak1989} (Corollary 2.4) the author showed
that a subquasivariety $\vv R$ of a congruence distributive variety is relative congruence distributive if and only if any finitely $\vv R$-irreducible algebra in $\vv R$ is finitely subdirectly irreducible in
the absolute sense.

But by Lemma \ref{all} every finitely $\vv Q$-irreducible algebra  has $p(x,y,z)$ as a dual i-discriminator term, so it is simple; hence it is finitely subdirectly irreducible in the absolute sense.
Therefore $\vv Q$ is relatively congruence distributive.
\end{proof}

\begin{remark} Let $\vv V$ be a dual i-discriminator variety with witness term $p(x,y,z)$. Which subquasivarieties of $\vv V$ are dual i-discriminator
with the same term?  In general not all of them; however let $\vv V_{fsi}$ be the class of finitely subdirectly irreducible (i.e. simple) algebras in $\vv V$.
Then $\QQ(\vv K)$ is a dual i-discriminator quasivariety with witness term $p(x,y,z)$ if and only if  $\vv K \sse \vv V_{fsi}$, so apparently we have a recipe to construct many dual i-discriminator quasivarieties. However it is an immediate consequence of Proposition 2.5 in \cite{Dziobiak1989} that if $\vv V$ is locally finite, then each of those subquasivarieties is really a variety.
\end{remark}

Dual i-discriminator variety have a special kind of RPIP; we say that
a quasivariety $\vv Q$ has the {\bf relative ternary principal intersection property} (RTPIP for short) if there is a ternary term $p(x,y,z)$ such that
$$
\cg_\alg A^\vv Q(a,b) \cap \cg_\alg A^\vv Q(a,b) = \cg_\alg A^\vv Q(p(a,b,c),p(a,b,d)).
$$

\begin{lemma}\label{RTPIP} Every dual i-discriminator variety $\vv Q$ has the RTPIP.
\end{lemma}
\begin{proof} Let $\alg A$ be a $\vv Q$-irreducible algebra in $\vv Q$ and let $a,b,c,d \in A$.
Then $p(x,y,z)$ is the dual i-discriminator function on $A$ therefore
$$
p(a,b,c) = p(a,b,d)\quad\text{if and only if}\quad\text{$a=b$ or $c=d$}.
$$
As $\vv Q$ is relatively congruence distributive, Theorem \ref{RPIP} yields the conclusion.
\end{proof}

The concept of having the RTPIP and being  a dual i-discriminator quasivariety  coincide for
finitely generated quasivarieties.

\begin{theorem}\label{cddualti} A  finitely generated quasivariety
 has  the RTPIP if and only if it is a  dual i-discriminator
quasivariety.
\end{theorem}
\begin{proof} Any dual i-discriminator variety  has the RTPIP by Lemma \ref{RTPIP}.
Suppose $\vv Q$ has the RTPIP, witness $p(x,y,z)$; then it is relative congruence distributive by Theorem \ref{RPIP}.

Moreover, since it is finitely generated, it has only finitely many $\vv Q$-irreducible algebras, all of them finite.
Let then $\alg A$ one of them and let $a,b \in A$ with $a \ne b$; then by Theorem \ref{RPIP} the function
$z \longmapsto t(a,b,z)$ is one-to one on $A$ and hence it is a permutation. The order of this permutation
divides $n!$, where $n$ is the largest among the cardinalities of $\vv Q$-irreducible algebras and the $n$-th power
of this permutation is the identity.

Define inductively
a sequence of terms $p_k(x,y,z)$ for $k=0,1,2,\dots,n-1$ by
$$
p_0(x,y,z) = z \qquad\qquad p_{k+1}(x,y,z) = p(x,y,p_k(x,y,z)).
$$
Then each $t_k$ for $1 \le k \le n$ witnesses the RTPIP for $\vv V$ and moreover  $t_n$
has the property that in each $\vv Q$-irreducible member of $\vv Q$,
$a\ne b$ implies  $t_n(a,b,c) = c$.
However, the unary term $\tau(x) = p_n(x,x,x)$ is not necessarily idempotent,
so further modifications are needed.

 For any $\vv Q$-irreducible algebra
$\alg A \in \vv Q$ there exist $l_\alg A,m_\alg A >0$ such that $\tau^{l_\alg
A+m_\alg A} = \tau^{l_\alg A}$ on $\alg A$. Let $L$ be a positive integer
such that for any $\vv  Q$-irreducible $\alg A \in \vv V$,
$L$ is both a multiple of $m_\alg A$ and greater than $l_\alg A$;
since there are only finitely many $\vv Q$-irreducible algebras in
$\vv Q$ we can take $L$ to be the product of all $l_\alg A$ and
$m_\alg A$.  Then on each $\vv Q$-irreducible member we have
$\tau^{2L} = \tau^{L}$, so the term $\pi(x) = \tau^L(x)$ is idempotent.

Now we must construct a dual idempotent discriminator term of which $\pi(x)$ is
the diagonal. Define terms $q_i(x,y,z)$ for $i=0,1,2,\dots,L$ by
$$
q_0(x,y,z) = z\qquad\qquad q_{k+1}(x,y,z) =
p_n(q_k(x,y,x),q_k(x,y,y),q_k(x,y,z)).
$$
Set $d = q_L$. Then $p(x,x,x) = \pi(x)$ as desired, while  $d$
retains the properties of $p_n(x,y,z)$.
Hence $p(x,y,z)$ is a dual i-discriminator term for $\vv Q$.
\end{proof}

\begin{corollary} \label{subirrsim} Let $\vv Q$ be a quasivariety with the
RTPIP; if $\alg A$ is a finite algebra in $\vv Q$, then
every $\QQ(\alg A)$ is dual i-discriminator variety.
\end{corollary}

Observe that Corollary \ref{subirrsim} is a handy tool for testing the
RTPIP for a  quasivariety. For instance, since a Heyting algebra is subdirectly irreducible if and only if the top element
is completely join irreducible, Heyting algebras have the RPIP with $q_1(x,y,z,w) := (x \leftrightarrow y) \join (z \leftrightarrow w)$ and
$q_2(x,y,z,w) \app 1$. However the only simple Heyting algebra is the two element one, so the only variety of Heyting
algebra with the RTPIP is the variety of Boolean algebras.

\subsection{Filtrality}

Let $\alg A \le_{sd} \prod_{i \in I} \alg A_i$; a congruence $\th \in \Con A$ is {\bf filtral} if there is a filter $F$ on $I$ such that
$$
\th = \{(a,b): \{i: a_i=b_i\} \in F\}.
$$
A quasivariety $\vv Q$ is {\bf relatively filtral} if, whenever $\alg A$ is subdirectly embeddable in a product of $\vv Q$-irreducible algebras, then every $\vv Q$-congruence of $\alg A$ is filtral.

\begin{theorem} \cite{CampercholiRaftery2017} Let $\vv Q$ be a quasivariety; then
\begin{enumerate}
\item $\vv Q$ has EDPRC if and only if it has DPRC and the RCEP;
\item $\vv Q$ is relatively filtral if and only if $\vv Q$ has EDPRC and every $\vv Q$-irreducible algebra  is $\vv Q$-simple.
\end{enumerate}
\end{theorem}

\begin{lemma} If $\vv Q$ is a dual i-discriminator variety, with witness term $p(x,y,z)$, then $\vv Q$ has DPRC and the CEP. In fact if $\alg A \in \vv Q$ and $a,b,c,d \in A$ then
$$
(c,d) \in \cg_\alg A^\vv Q(a,b)\quad\text{if and only if}\quad \forall u [ p(c,d,u) = p(c,d,p(a,b,u))].
$$
\end{lemma}
\begin{proof} By Lemma \ref{semi} and Theorem \ref{thm: join coprincipal},
$\cg^\vv Q_\alg A(a,b)$ and $\g^\vv Q_\alg A(a,b)$ are complements, thus
$$
 \cg^\vv Q_\alg A(c,d) \le \cg^\vv Q_\alg
A(a,b)\qquad\text{if and only if}\qquad\g^\vv Q_\alg A(a,b) \le \g^\vv Q_\alg A(c,d).
$$
By the definition of $\g^\vv Q_\alg A(a,b)$ the latter is equivalent to
\begin{equation}\label{eq5}
\forall u,v\ \ p(a,b,u) = p(a,b,v) \quad \Longrightarrow\quad
p(c,d,u) = p(c,d,v).
\end{equation}
Now assume the latter and let $v = p(a,b,u)$. Since the equation
$$
p(x,y,p(x,y,z)) \app p(x,y,z)
$$
holds in $\vv Q$ we get
$$
p(a,b,u) =  p(a,b,p(a,b,u)) = p(a,b,v);
$$
therefore
$$
p(c,d,u) = p(c,d,v) = p(c,d,p(a,b,u)).
$$
Conversely, if for all $u$, $p(c,d,u) = p(c,d,p(a,b,u))$ and $p(a,b,u) =
p(a,b,v)$, then
$$
p(c,d,u) = p(c,d,p(a,b,u)) = p(c,d,p(a,b,v)) = p(c,d,v).
$$
Therefore that $\vv Q$ has DPRC. Also  observe that the formula
defining principal $\vv Q$-congruences in $\vv Q$ is universal and so it is preserved
by subalgebras. Hence if $\alg A \in \vv Q$ any principal
$\vv Q$-congruence  of a subalgebra $\alg B$ of $\alg A$ extends to a principal
$\vv Q$-congruence of $\alg A$. This is enough to guarantee the RCEP \cite{BlokPigozzi1997}.
\end{proof}

\begin{corollary} \label{did relatively filtral} Every dual i-discriminator quasivariety is relatively filtral.
\end{corollary}

We will close this section with some examples.

\begin{example} There are filtral varieties without the RTPIP, which shows that the implication in Corollary \ref{did relatively filtral}
cannot be reversed.

A {\bf de Morgan algebra}
is an algebra $\la A,\join,\meet,\overline{\phantom{a}},0,1\ra$ in which
the $\{\meet,\join,0,1\}$-reduct is a bounded distributive lattice and
moreover
$$
\overline{x \join y} = \overline{x}\meet\overline{y}\qquad
\overline{x \meet y} = \overline{x}\join\overline{y}\qquad
\overline{\overline{x}} = x.
$$
The only  subdirectly irreducible de Morgan
algebras are shown in Figure \ref{demorgan} (see \cite{Kalman1958}) and they are all simple; moreover the variety $\vv M$ of de
Morgan algebras
has EDPC \cite{EDPC1} and hence it is filtral.
\medskip

\begin{figure}[htbp]
\begin{center}
\begin{tikzpicture}[scale=.9]
\draw (0,0) --(0,1);
\draw (3,0) -- (3,2);
\draw  (7,0) -- (6,1) -- (7,2) -- (8,1) --(7,0);
\draw[fill] (0,0) circle [radius=0.05];
\draw[fill] (0,1) circle [radius=0.05];
\draw[fill] (3,0) circle [radius=0.05];
\draw[fill] (3,1) circle [radius=0.05];
\draw[fill] (3,2) circle [radius=0.05];
\draw[fill] (7,0) circle [radius=0.05];
\draw[fill] (6,1) circle [radius=0.05];
\draw[fill] (7,2) circle [radius=0.05];
\draw[fill] (8,1) circle [radius=0.05];
\node[right] at (0,0) {\footnotesize  $0=\ol{1}$};
\node[right] at (0,1) {\footnotesize  $1=\ol{0}$};
\node[right] at (3,0) {\footnotesize  $0=\ol{1}$};
\node[right] at (3,1) {\footnotesize  $a=\ol{a}$};
\node[right] at (3,2) {\footnotesize  $1=\ol{0}$};
\node[right] at (7,0) {\footnotesize  $0=\ol{1}$};
\node[left] at (6,1) {\footnotesize  $a=\ol{b}$};
\node[right] at (8,1) {\footnotesize  $b=\ol{a}$};
\node[right] at (7,2.2) {\footnotesize  $1=\ol{0}$};
\node[below] at (0,-.2) {\footnotesize  $\alg{M}_2$};
\node[below] at (3,-.2) {\footnotesize  $\alg{M}_3$};
\node[below] at (7,-.2) {\footnotesize  $\alg{M}_4$};
\end{tikzpicture}
\end{center}
\caption{Simple de Morgan algebras}\label{demorgan}
\end{figure}
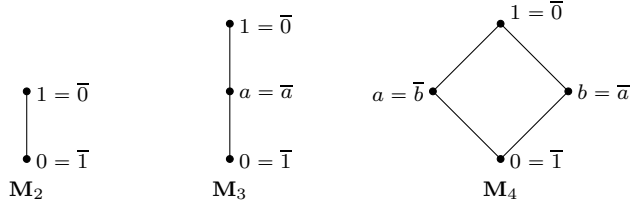

The variety $\vv K$ of {\bf Kleene algebras} is the variety generated by $\alg M_3$ and we will show that $\vv K$ does not have the TPIP.  From now on, let us  denote $\join$ and $\meet$
by $+$ and juxtaposition, in order  to take advantage of the familiar
linear algebra notation.  By
using the fact that the operation $x \mapsto \bar x$ is a dual
automorphism, we may write any unary
polynomial function $p(x)$ of $\alg A \in \vv M$
 in the form $p(x) = \alpha x \bar x
+ \beta x + \gamma \bar x + \delta$ for $\alpha, \beta, \gamma, \delta
\in A$.
\bigskip

\noindent{\bf Claim.}
{\it Let $p(x)$ be a unary polynomial on $\alg M_3$; then either $p(x)$ has a
constant Boolean value ($0$ or $1$) or else $p(a)= a$. Hence
if $p(x)$ is one-to-one on $\alg M_3$, then either $p(x) = x$
or $p(x) = \bar x$; in particular $p(x)$ takes Boolean elements ($0$ and $1$)
to Boolean elements.}
\medskip

To prove this claim, observe that
$$
p(a) = \a a\bar a +\b a + \g \bar a + \d = (\a + \b + \g) a + \d
$$
for some $\a,\b,\g,\d \in \alg M_3$. If $\d = 1$, then $p(x)$ is constantly
$1$; if $\d = a$, then $p(x)$ is constantly $a$. Finally if
$\d = 0$, then $p(x)$ is constantly $0$ if $\a + \b + \g = 0$ and
otherwise $p(a) = a$.
For the second assertion of the claim, note that $p(a) = a$ and, since $p(x)$
is one-to-one, it must take $0,a,1$ either to $0,a,1$ or $1,a,0$.
\qed
\medskip

Now suppose that $\vv K$ has the RTPIP with ternary intersection term $t(x,y,z)$. Then the polynomial $p(z) = t(0,a,z)$ is
one-to-one on $\alg M_3$ and hence by the Claim
$p(0) = t(0,a,0)$ is Boolean, i.e. $0$ or $1$.
Now consider the polynomial $q(y) = t(0,y,0)$; since
$q(a)= t(0,a,0) \ne a$, then $q(y)$ is constant giving
$t(0,0,0) = q(0) = q(1) = t(0,1,0)$.
Similarly, $t(0,0,1) = t(0,1,1)$. Since
$q'(z) = t(0,0,z)$ is assumed to be constant we  also have
$t(0,0,0) = t(0,0,1)$. Putting these facts together we get
$$
t(0,1,0) = t(0,0,0) = t(0,0,1) = t(0,1,1),
$$
which contradicts the assumption that $t(x,y,z)$ is a ternary intersection
term.\qed
\end{example}

\begin{remark}  The variety $\vv K$ of Kleene algebras does have the
RPIP, with witness terms $q_1(x,y,z,w) =h(x,y) + h(z,w)$, where
$h(x,y) = x \bar xy\bar y$, and $q_2(x,y,z,w) = q_1(x,y,z,w) + (x + y)(\bar x
+ \bar y)(z + w) (\bar z + \bar w)$.
To see this first observe that
the only subdirectly irreducibles in $\vv K$ are
$\alg M_3$ and the two-element Boolean algebra $\alg M_2$. However,
since $\alg M_2$ is a subalgebra of $\alg M_3$ it is enough to consider
the latter. We have hence to show that for any  $r,s,u,v \in \alg M_3$
\begin{equation}
p(r,s,u,v) = q(r,s,u,v)\qquad\text{if and only if}\qquad r = s\quad \text{or}\quad u =
v. \tag{P}
\end{equation}
First observe that
$$
h(r,s) = \begin{cases} 0\quad\text{if $(r,s) \ne (a,a)$}\\
a\qquad\text{otherwise} \end{cases}.
$$
Suppose then that either $(r,s) = (a,a)$ or $(u,v) = (a,a)$ in
(P). Then $p(r,s,u,v) = a = q(u,v,r,s)$, since
the summands in the definition of $q$ have values $a$ and at most $a$.

Otherwise, $h(r,s)=h(u,v) = 0$ so (P) becomes
$$
0 = (r +s)(\bar r + \bar s)(u + v)(\bar u + \bar v);
$$
since $0$ is meet irreducible this occurs when some  factor
is $0$. But this happens exactly when $r=s$ or $u = v$, which is precisely
what (P) asserts.

 It can be shown that the variety $\vv M$ of de Morgan
algebras does not have the RPIP. We will not prove it but we observe that  the above reasoning
fails in  $\vv M$: in  $\alg M_4$,  (P) is contradicted
by
$$
p(0,a,0,b) = 0 = q(0,a,0,b).
$$
\end{remark}

\begin{example}\label{realring} There are also quasivarieties with the RTPIP that fail to be filtral, showing that finite generation cannot be removed from the  hypotheses of Corollary \ref{subirrsim}.
Let $\alg A = \la \mathbb R,+,-,\cdot,0,\join,\meet,\sigma\ra$, where $\sigma$ is the unary operation
given by $\sigma(r) = \sqrt{|r|}$ for $r \in \mathbb R$. Thus $\alg A$
is a lattice ordered ring with an additional operation $\sigma$ and it is easy to check that  in
$\VV(\alg A)$ the congruences are determined by the $\sigma$-ideals,
i.e. the convex ring ideals such that $a-b \in I$ implies
$\sigma(a)-\sigma(b) \in I$.

Let now $\vv Q=\QQ(\alg A)$ and  $t(x,y,z) = (x-y)z$; to prove that $t(x,y,z)$ witnesses the RTPIP for $\vv Q$ it is enough to show that every algebra in $\II\SU\PP_u(\alg A)$ is an integral
domain, regarded as a ring (because of Theorems \ref{quasivariety} and \ref{RPIP}). But this is obvious, since $\alg A$ is an integral domain and we can express ``integral domain''  via a universal sentence.

To show that $\vv Q$ is not relatively filtral we observe first that in a relative filtral quasivariety $\vv R$, if an algebra $\alg B\in \vv R$ is such that $\op{Con}_\vv R(\alg B)$ is totally ordered, then
$\alg B$ is $\vv R$-simple.  Indeed, since $\vv R$-irreducible algebras are $\vv R$-simple,
every strictly meet irreducible $\vv R$-congruence of $\alg B$  is a
coatom in $\op{Con}_\vv R(\alg B)$ and since $\op{Con}_\vv R(\alg  B)$ is totally ordered there can be only one coatom.
But $0_\alg B \in \op{Con}_\vv R(\alg B)$  and it is the meet of strictly meet irreducible congruences
so it is itself a coatom and thus $\alg B$ is $\vv R$-simple.

We will produce an algebra in $\vv Q$ with a totally ordered lattice of $\vv Q$-congruences that it is not simple.  Let $\alg A^*$ be any ultraproduct of $\alg A$; then the universe of $\alg A^*$ is a nonstandard
model $\mathbb R^*$ of the reals and since $\alg A^*$ is totally ordered, then  $\op{Con}_\vv Q(\alg A^*)$ is totally ordered. Let $F$ be all the elements of $\mathbb R^*$ that are not infinite; it is trivial to check that $F$ is the universe of a subalgebra $\alg F$ of $\alg A^*$.  We claim that
if $I$ is the set of infinitesimals in $F$ is a $\sigma$-ideal of $\alg F$; this happens because the continuity of $\sigma$ as a real-valued function ensures that if $a$ and $b$ are infinitely close in $F$, then so are $\sigma(a)$ and $\sigma(b)$.  Let $\th_I$ be the congruence associated to $I$; then $\th_I$ collapses all the infinitesimals to $0$, so $\alg F/\th_I \cong \alg A$ and $\th_I \in \op{Con}_\vv Q(\alg F)$. As $\th_I$ is neither the largest or the smallest $\vv Q$-congruence of $\alg F$ we have proved out claim.\qed
\end{example}

\subsection{Fixedpoint discriminator quasivarieties}

In this section we will explore an important subclass of dual i-discriminator quasivarieties. A pointed quasivariety $\vv Q$ (and from now on we will denote constant by $0$) is a
{\bf fixedpoint discriminator quasivariety} if there is a ternary term $d(x,y,z)$ such that, for any $\vv Q$-irreducible algebra $\alg A \in \vv Q$ and for any $a,b,c \in A$
$$
t(a,b,c) = \left\{
             \begin{array}{ll}
               c, & \hbox{if $a=b$;} \\
               0, & \hbox{if $a \ne b$.}
             \end{array}
           \right.
$$

\begin{theorem}\label{fixpoint discriminator} For a pointed quasivariety $\vv Q$ the following are equivalent:
\begin{enumerate}
\item $\vv Q$ is a fixedpoint discriminator variety;
\item $\vv Q$ is a dual i-discriminator variety with dual i-discriminator term $p(x,y,z)$ satisfying
$p(x,x,x) \app 0$.
\end{enumerate}
\end{theorem}
\begin{proof} Assume (1) and let $p(x,y,z) = d(0,d(x,y,z),z)$; then for any $\alg A \in \vv Q_{ir}$ and $a,b,c \in A$
$$
p(a,b,c) = \left\{
             \begin{array}{ll}
               0, & \hbox{if $a=b$;} \\
               c, & \hbox{if $a\ne b$.}
             \end{array}
           \right.
$$
Then $p(a,a,a) = 0$ and (2) holds.

Conversely assume (2) and let $d(x,y,z) = p(z,p(x,y,z),z)$. Let $\alg A \in \vv Q_{ir}$
and $a,b,c \in A$; then if $a=b$
$$
d(a,a,c) = p(c,p(a,a,c),c) = p(c,0,c) = c
$$
and if $a \ne b$
$$
d(a,b,c) = p(c,p(a,b,c),c) =  p(c,c,c) = 0.
$$
Therefore $d(x,y,z)$ is a fixedpoint discriminator for $\vv Q$.
\end{proof}

\begin{example}\label{implalg} (Implication algebras) An {\em implication algebra} is
an algebra with a single binary operation $\imp$ axiomatized by
\begin{align*}
&(x \imp y) \imp x \approx x \\
&(x \imp y) \imp y \approx (y \imp x)\imp x\\
& x \imp (y\imp z) \approx y \imp (x \imp z).
 \end{align*}
Equivalently it is an  $\imp$-subreduct of a Boolean algebra.
The only subdirectly irreducible implication algebra is the simple algebra
$\la \{0,1\},\imp\ra$, where $\imp$ is the Boolean implication. A dual
i-discriminator on $\la \{0,1\},\imp\ra$ with respect to $1$ is
$$
p(x,y,z) = [(x\imp y) \imp ((y\imp x) \imp z)] \imp z.
$$
and of course $p(x,x,x)=1$.  So the variety of implication algebras is
a fixedpoint discriminator variety by Theorem \ref{fixpoint discriminator}.
Implication algebras are not congruence permutable \cite{Mitschke1971},
hence they are not a discriminator variety. Moreover  one shows by induction
that, for any ternary term $t(x,y,z)$, at least
one of $t(0,0,1)$, $t(0,1,0)$, $t(1,0,0)$ evaluates to $1$ in $\la
\{0,1\},\imp\ra$ \cite{EDPC1}. Hence the variety of implication algebras
is not a dual discriminator variety.\qed
\end{example}

\begin{remark} In \cite{EDPC3} the term {\em fixedpoint discriminator variety} refers to a slightly different concept.
The {\em fixedpoint discriminator function} on a set $A$ is a ternary function $t(x,y,z)$ with the property that there is a $t_A \in A$ (called the {\bf discriminating element})
such that for any $a,b,c \in A$
$$
t(a,b,c) = \left\{
             \begin{array}{ll}
               c, & \hbox{if $a =b$;} \\
               t_A, & \hbox{if $a=b$.}
             \end{array}
           \right.
$$
A {\em fixedpoint discriminator variety} in the sense of \cite{EDPC3} is a variety having a fixedpoint discriminator function on each subdirectly irreducible algebra.

Let us show that the concepts differ. Let $\alg Q = \la \{0,1\},q\ra$; we regard $\alg Q$
as a term reduct of the two-element Boolean algebra, where
$$
q(x,y,z) = x \meet ((y\meet z) \join (y' \meet z')).
$$
By Theorem \ref{birkhoff}, $\alg Q$ is the unique subdirectly
irreducible member of $\VV(\alg Q)$ and for any $a,b,c \in Q$
$$
q(a,b,c)=\begin{cases} a&\text{if $b=c$}\\ 0&\text{if $b\ne c$.}\end{cases}.
$$
Thus $t(x,y,z) = q(z,x,y)$ is a fixedpoint discriminator function for $\alg
Q$ with discriminating element $0$ and $\VV(\alg Q)$ is a fixedpoint
discriminator variety in the sense of \cite{EDPC3}. It is easily seen that $\VV(\alg Q)$ has the RPIP
witnessed by
$$
q_1(x,y,z,w)=q(x,y,q(x,z,w))\qquad q_2(x,y,z,w)=q(y,x,q(y,z,w)).
$$
However $\VV(\alg Q)$ does not have the RTPIP (\cite{AglianoBaker1999} Proposition 2.6) so it cannot be a fixedpoint discriminator variety.

 The problem is that the discriminating element $0$ of $\alg Q$ above is not induced
by any constant term of the variety. If we add it to type i.e., we consider
$\VV(\la\{0,1\},q,0\ra)$, then this variety is termwise equivalent to
the variety of implication algebras.

This is not a coincidence; it is easily seen that a pointed (quasi)variety
 is a fixedpoint discriminator variety if and only if
it is a fixedpoint discriminator variety in the sense of \cite{EDPC3} in which the discriminating element
is a constant.  However the latter is quite a mouthful, so we will stick to our name.
\end{remark}

A fixedpoint discriminator variety is not merely
a variety with a constant arbitrarily designated: there is no ternary term
that is the  fixedpoint discriminator with constant $0$ or $1$ in the
variety of bounded distributive lattices (which is a dual discriminator variety).
Rather we will see  that ``local'' properties near $0$ in a member of a  fixedpoint discriminator variety
determine its ``global'' properties.

For classical pointed structures, like groups, rings or
Boolean algebras, one often deals with subsets (like normal subgroups, ideals
or filters) that have two properties:
\begin{enumerate}
\ib   they are the congruence blocks of some constant;
\ib   they are definable by certain ``closure''  properties.
\end{enumerate}
For instance a subset $I$ of a commutative ring $\alg R$ is an ideal if and
only if $r \in R$ and $a \in I$ implies $ra \in I$. Using
the classical description of terms in commutative rings,  $I$ is an ideal
of $\alg R$ if and only if for any term $t(\vuc xm,\vuc yn)$ in the language
of rings such that
\begin{equation}
t(\vuc xm,0,\dots,0) \approx 0,\tag{$**$}
\end{equation}
then for any $\vuc rm \in R$ and $\vuc an \in I$,
$$
t(\vucc rm,\vuc an) \in I.
$$
It is a classical theorem in ring theory that it is sufficient to check
the above for the terms $y_1-y_2$ and $x_1y_1$.

Let now $\vv K$ be any class of algebras whose type contains a  $0$.
 A term $p(\vuc xm,\vuc yn)$ is a {\bf $\vv K$-ideal term in $\vec y$} \cite{OSV1}  if the identity
$p(\vec x,0,\dots,0) \ap 0$ holds in $\vv K$.
A nonempty subset $I$ of $\alg A
\in \vv K$ is a {\bf $\vv K$-ideal} of $\alg A$ if for any $\vv K$-ideal term $p(\vec x,\vec y)$, for $\vec a \in A$ and $\vec b \in I$, $p(\vec a,
\vec b)\in I$. Under inclusion,
the set $\op{I}_{\vv K}(\alg A)$ of all $\vv K$-ideals of $\alg A$ is an
algebraic lattice; if $H \sse A$, the $\vv K$-ideal generated by $H$
is easily seen to be the set
$$
I_\vv K( H) =\{p(\vec a,\vec b) : p(\vec x,\vec y)\ \text{a $\vv K$-ideal term}, \vec a \in A, \vec b \in H\}.
$$
Moreover one easily checks that for all $\th \in \Con A$, $0/\th\in \op{I}_\vv K(\alg A)$ and that the map
from $\Con A$ into $\op{I}_{\vv K}(\alg A)$ defined by
$\th \longmapsto 0/\th$ is a lattice homomorphism.

In general being an ideal is not an absolute concept and depends on the class $\vv K$ we are considering; moreover in general the above homomorphism
is not onto, i.e. there are ideals which are not $0$-classes of any congruence. However there is a sufficient condition for this to happen.

Let $\vv K$ be any class of algebras whose type contains a  $0$; {\bf subtraction term} for $\vv K$ is a binary term $s(x,y)$ such that $\vv K$
satisfies
$$
s(x,x) \app 0\qquad s(x,0) \app x.
$$
In this case we say that $\vv K$ is {\bf $s$-subtractive} or simply {\bf subtractive}.  The key point is the following observation.

\begin{lemma} (\cite{OSV1}, Proposition 1.4) Let $\vv K$ be a subtractive class  and let $\alg A \in \vv K$; then every $\vv K$-ideal of $\alg A$ is the $0$-class of a congruence of $\alg A$.
\end{lemma}

Therefore being a $\vv K$-ideal of an algebra in a subtractive class is an {\em absolute} concept and we can drop the decoration $\vv K$ and simply talk about ideals of $\alg A$. Since
any subtractive class generates a subtractive variety, really the theory of ideals in subtractive classes is the theory of ideals in subtractive {\em varieties}. This theory has been developed
at length in the 1990 by A. Ursini and the first author (see \cite{OSV4} and the bibliography therein). Here we simply quote a general description, whose proof can be found in \cite{OSV1}.

\begin{theorem}\label{sub} For a variety $\vv V$ the following are equivalent:
\begin{enumerate}
\item $\vv V$ is subtractive;
\item for any algebra $\alg A \in \vv V$ and any $\th,\f \in \Con A$
$$
0/(\th \join \f) = 0/(\th \circ \f),
$$
where $0/(\th \circ\f) = \{a \in A: \text{exists a}\ b \in A\ \text{with}\
0\mathrel{\th}b\mathrel{\f}a\}$;
\item for any $\alg A \in \vv V$ the mapping $\th \longmapsto 0/\th$ from $\Con A$ to $\op{I}(\alg A)$ is a surjective and
complete lattice homomorphism.
\end{enumerate}
\end{theorem}

It is obvious that if a subtractive variety $\vv V$ is also point regular at $0$ in the classical sense, then for any $\alg A \in \vv V$ the epimorphism from $\Con A$ to $\op{I}(\alg A)$ is in fact
an isomorphism; such varieties are usually called {\bf ideal determined} \cite{OSV1}.

\begin{remark} Let $\alg A$ be the algebra in Example  \ref{realring}. Then $\VV(\alg A)$ is a good
example of this behavior. In fact $\VV (\alg A)$ is a variety of $\ell$-rings with operators ring subtraction witness at the same
time point-regularity at $0$ and subtractivity. So $\VV(\alg A)$ is ideal determined and  a $\sigma$-ideal of any
$\alg B\in \VV(\alg A)$ is simply an ideal in our sense.
\end{remark}

We observe that while subtractivity of a class implies subtractivity for the variety generated by that class this is not true
for point regularity. For instance the quasivariety of $\mathsf{BCK}$-algebras is point regular at $1$ but generates a variety that is not point regular at $1$ (see Section 4.5 in \cite{OSV3}).

A quasivariety $\vv Q$ (whose type contains a constant $0$) is {\bf relatively $0$-regular} if for all $\alg A \in \vv Q$ and for all $\th,\f \in \op{Con}_\vv Q(\alg A)$, if $0/\th =0/\f$, then $\th = \f$. The proof of the following theorem is easily patterned after the analogous result for varieties.

\begin{theorem}\label{regular} For a quasivariety $\vv Q$  whose type contains  a constant $0$
the following are equivalent;
\begin{enumerate}
\item  $\vv Q$ is relatively $0$-regular;
\item  there are binary terms $\vuc rn$ such that the equivalence
$$
r_1(x,y) = \dots = r_n(x,y) = 0 \qquad \text{if and only if} \qquad x = y
$$
holds in $\vv Q$;
\item  for any algebra $\alg A \in \vv Q$ and $\th \in \op{Con}_\vv Q(\alg A)$
$$
0/\th = \{0\}\qquad\text{if and only if}\qquad \th = 0_\alg A.
$$
\item  for any algebra $\alg A \in \vv Q$ and   $a,b \in A$
$$
0/\cg^\vv Q_\alg A(a,b) = \{0\} \qquad\text{if and only if}\qquad a = b.
$$

\end{enumerate}
\end{theorem}

 Now let $\vv Q$ be relatively $0$-regular; if $\alg A \in \vv Q$ and $I = 0/\th$ for some $\th \in \op{Con}_\vv Q(\alg A)$, then $\th$ is the only $\vv Q$-congruence with that property. So we can define a {\bf $\vv Q$-ideal} to be an ideal of $\alg A$ that is
the $0$-block of a (necessarily unique) $\vv Q$-congruence of $\alg A$.  Now, as the intersection of a family of $\vv Q$-ideals is a $\vv Q$-ideal, $\vv Q$-ideals form an algebraic lattice
which we denote by $\op{I}_\vv Q(\alg A)$.  The following lemma can be obtained  by combining the results in Section 5 and 6 of
\cite{BlokRaftery1998}, keeping in mind that,  in a  relatively $0$-regular quasivariety, the {\em strong ideals} in the sense of \cite{BlokRaftery1998}  are $0$-classes of $\vv Q$-congruences and hence $\vv Q$-ideals in our sense.

\begin{lemma}\label{relativehom} If $\vv Q$ is relatively $0$-regular and $\alg A \in \vv Q$ then the mapping $\th \longmapsto 0/\th$ is a lattice isomorphism from $\op{Con}_\vv Q(\alg A)$ to $\op{I}_\vv Q(\alg A)$.
\end{lemma}

To prove the main theorem of this section we need to recall the theory of definability of principal ideals developed in \cite{OSV4}, adapted to subtractive quasivarieties.  A subtractive quasivariety $\vv Q$ has  {\bf equationally definable principle relative ideals} (EDPRI) if there is an $n \in \mathbb N$ and binary terms $d_i,e_i$, $i=1,\dots,n$ such that for all $\alg A \in\vv Q$ and $a,b \in A$
$$
a\in I_\vv Q(b)\qquad\text{if and only if}\qquad d_i(a,b) = e_i(a,b),\   i =1,\dots,n.
$$
The following lemma is a straightforward generalization of   Theorem 2.2 in \cite{OSV4}.

\begin{lemma}\label{EDPRC EDPRI} If a subtractive quasivariety $\vv Q$ has EDPRC, then it has EDPRI; if $\vv Q$ is relatively ideal determined, then it has
EDRPC if and only if it has EDPRI.
\end{lemma}

In \cite{OSV4} it is shown that in subtractive variety with EDPRI we can always reduce the number of terms witnessing  EDPRI to one; actually one easily sees, by looking at the proofs of the relevant parts of Theorems 2.6, 3.1 and 3.4 in \cite{OSV4},  that the conclusions hold for any subtractive class that is closed under $\SU$ and $\PP$. In particular they hold for quasivarieties and relative ideals so we can formulate the following:

\begin{lemma}\label{term} Let $\vv Q$ be a subtractive quasivariety with EDPRI; then there exists a binary term $u(x,y)$ such that for all $\alg A \in \vv Q$ and $a,b \in A$
\begin{align*}
&u(a,a) =0\\
&u(a,0) = a\\
&u(0,a)= 0\\
& a \in I_\vv Q(b)\qquad\text{if and only if}\qquad  u(a,b) = 0.
\end{align*}
\end{lemma}

It follows at once that if $\alg A$ is a $\vv Q$-simple algebra in $\vv Q$, then for all $a,b \in A$
\begin{equation}
u(a,b) = \left\{
           \begin{array}{ll}
             0, & \hbox{if $b\ne 0$;} \\
             a, & \hbox{ if $b=0$.}
           \end{array}
         \right. \tag{S}
\end{equation}

\begin{theorem} \label{rid rfi} For a pointed quasivariety $\vv Q$ the following are equivalent:
\begin{enumerate}
\item  $\vv Q$ is a fixedpoint discriminator quasivariety;
\item  $\vv Q$ is relatively ideal determined and relatively filtral.
\end{enumerate}
\end{theorem}
\begin{proof} Assume (1); then $\vv Q$ is relatively filtral by Corollary \ref{did relatively filtral}, hence every $\vv Q$-irreducible algebra is $\vv Q$-simple. Moreover if $d(x,y,z)$ is the  fixedpoint discriminator for $\vv Q$, then
$s(x,y) := d(0,d(x,y,x),x)$  is a subtraction term  for $\vv Q$.

To prove relative $0$-regularity, let $\alg A \in \vv V$ and $a,b \in \alg A$
with $0/\cg^\vv Q_\alg A(a,b) = \{0\}$.  We observe that by  Theorems \ref{thm: join coprincipal} $\cg^\vv Q_\alg A(a,b)$ and $\g^\vv Q_\alg A(a,b)$ are complements in $\op{Con}_\vv Q(\alg A)$ and
by Lemma  \ref{relativehom}
$$
0/\g^\vv Q_\alg A(a,b) = 0/\cg^\vv Q_\alg A(a,b) \join 0/\g^\vv Q_\alg A(a,b) = 0/(\cg^\vv Q(a,b) \join \g^\vv Q(a,b) = 1_\alg A = A,
$$
where the joins are taken in $\op{I}_\vv Q(\alg A)$ and $\op{Con}_\vv Q(\alg A$).
It follows that $\g^\vv Q_\alg A(a,b) = 1_\alg A$ and hence, again by Theorem \ref{thm: join coprincipal} $\cg^\vv Q_\alg A(a,b) = 0_\alg A$, so that $a=b$.
By Theorem \ref{regular}(4) we conclude that $\vv Q$ is relatively $0$-regular.

Conversely assume that $\vv Q$ is relative ideal determined and relatively filtral; so $\vv Q$ is subtractive, relatively $0$-regular,  every $\vv Q$-irreducible algebra is simple and has EDPRC. Hence
by Lemma \ref{EDPRC EDPRI} it has EDPRI and Lemma \ref{term} applies.  Let $\vuc rn$ be the terms witnessing
relative $0$-regularity for $\vv Q$; define
$$
d(x,y,z) = u(r_1(x,y),u(r_2(x,y),u(\dots u(r_n(x,y),z) \dots ))).
$$
Now let $\alg A$ be a $\vv Q$-irreducible (i.e., simple) algebra
in $\vv Q$ and let $a,b, c \in A$. If $a = b$, then $r_i(a,b) = 0$
for $i=1,\dots,n$. Hence by (S), $d(a,b,c) = c$.
If $a \ne b)$, then there is a largest $i \le n$ such that
$r_i(a,b) \ne 0$; hence, again by (S), $u(r_i(a,b),c) = 0$
and so eventually $d(a,b,c) = 0$. So $d(x,y,z)$ is a fixedpoint discriminator term on any simple algebra in $\vv Q$ and thus
$\vv Q$ is a fixedpoint discriminator variety.
\end{proof}

\begin{remark} The theory of ideals in quasivarieties, as well as the theory of subtractive (quasi)varieties can be expanded further in several directions (see \cite{BlokRaftery1998} or \cite{KowalskiPaoliSpinka2011}).
A less known path was explored by G. Barbour but never published\footnote{this has been communicated to the first author by J.G.Raftery and other South African colleagues.}.   Let $\vv K$ be a class of algebras and $e(x)$ a unary term. A class $\vv K$ is an {\bf $e$-Mal'cev class} if there are a unary idempotent term $e(x)$ and a ternary term $m_e(x,y,z)$ of $\vv K$ such that
$$
m_e(x,x,e(y)) \app e(y) \qquad m_e(x,e(y),e(y)) \app x
$$
hold in $\vv K$.  If $\alg A\ \in \vv K$
we say that {\bf the congruences permute at $e$} if for each $a \in A$
and $\a,\b \in \Con A$
$$
e(a)\mathrel{\a} \circ \mathrel{\b} b\qquad\text{if and only if}\qquad
e(a)\mathrel{\b} \circ \mathrel{\a} b.
$$
We can prove in a standard way that the following are equivalent:
\begin{enumerate}
\ib $\vv V$ is an $e$-Mal'cev variety;
\ib for every $\alg A$ the congruences of $\alg A$ permute at $e$.
\end{enumerate}
So e-Mal'cev varieties generalize both congruence permutable varieties (take $e(x):=x$) and
subtractive varieties (take $e(x) =0$).  The whole machinery of ideal term and ideals can be transferred any $e$-Mal'cev class $\vv K$, and similar theorems can be proved with only
straightforward changes; for instance an {\bf $e$-ideal term} $t(\vec x,\vec y)$ is a term such that the equation $t(\vec x, \vec{e(y)}) \app e(y)$ holds in $\vv K$. Then if $\alg A \in \vv K$ an
 {\bf $e$-ideal} of $\alg A$ is a subset $I \sse A$ that is ``closed'' for all the $e$-ideal terms. It is obvious that if $u \in A$ and $\th \in \Con A$, then
$e(a)/\a$ is an $e$-ideal of $\alg A$. If for $u\in A$ we let $\op{I}_{e(u)}(\alg A)$ denote the set of all
$e$-ideals of $\alg A$ containing $e(u)$ (briefly the {\bf $e(u)$-ideals}),
then $\op{I}_{e(u)}(\alg A)$ is an
algebraic lattice for any $u\in A$.

A quasivariety $\vv Q$ is {\bf relatively $e$-regular} if there are ternary terms $r_i(x,y,z)$, $i=1,\dots, n$
such that for any $\alg A \in \vv Q$ and $a,b,c \in A$
$$
r_1(a,b,c)= \dots =r_n(a,b,c) =  e(c)\qquad\text{if and only if}\qquad  a=b.
$$
Not surprisingly:
\begin{enumerate}
\ib if $\vv Q$ is an $e$-Mal'cev class $\alg A \in Q$ and $u \in $, then for every $I\in \op{I}_{e(u)}$ there is a congruence $\th \in \op{Con}(\alg A)$ such that $I= e(u)/\th$;
 \ib is $\vv Q$ is $e$-Mal'cev and relative $e$-regular then the for any $u \in A$ the mapping above is an isomorphism from $\op{Con}_\vv Q(\alg A)$ to the lattice of the $\vv Q$-$e(u)$-ideals, i.e. the $e(u)$-ideals ``coming from''  $\vv Q$-congruences.
\end{enumerate}
Really the entire content of this section can be generalized to $e$-Mal'cev varieties; let's say that $\vv Q$ is an {\bf i-discriminator quasivariety} if there are a unary idempotent term $e(x)$ and a ternary term $d_e(x,y,z)$ such that for any $\vv Q$-irreducible algebra $\alg A$ and for any $a,b,c \in A$
$$
d_e(a,b,c) = \left\{
               \begin{array}{ll}
                 e(a), & \hbox{if $a\ne b$;} \\
                 c, & \hbox{if $a=b$.}
               \end{array}
             \right.
$$
Then;
\begin{enumerate}
\ib  every i-discriminator quasivariety is a $e$-Mal'cev dual i-discriminator variety;
\ib  a quasivariety is a i-discriminator variety if and only if it is $e$-Mal'cev, relatively $e$-regular and relatively filtral.
\end{enumerate}
The proofs use some general results in \cite{BlokRaftery1998} and  involve the generalization of several results in \cite{OSV3}; those generalizations are more or less straightforward but more calculations are involved and we felt that we did not need this kind of generality.
\end{remark}

\subsection{$C$-completeness and the dual i-discriminator}

Every dual i-discriminator quasivariety $\vv Q$ satisfies certain equations and quasiequations involving the dual i-discriminator term. If the quasivariety is relative congruence distributive of finite type (so for instance if it is a congruence distributive variety of finite type) then it can be characterized purely by equations involving the dual i-discriminator term and the fundamental
operations, even though the  definition involves non-equational properties of $\vv Q$-irreducible members.

\begin{theorem} Let $\tau$ be a finite type of algebras and let $p(x,y,z)$ be a term of type $\tau$; let also $\vv K$ be the class of algebras of type $\tau$ on which
$p(x,y,z)$ is a dual i-discriminator term. Then there is a finite set of equations $\Gamma$ such that $\QQ(\vv K)$ is a dual i-discriminator quasivariety (with witness term  $p(x,y,z)$) if and only if $\VV(\vv K)\vDash \Gamma$.
\end{theorem}
\begin{proof} For (1), by Corollary \ref{all} $\QQ(\vv K)$ is itself a dual idempotent discriminator variety and $\vv K$ consists exactly of all
the $\QQ(\vv K)$-irreducible algebras. Note that $\vv K$ is strictly elementary since it is defined by the formula $\sigma$ in the proof of Lemma \ref{all}
and it is  also relative congruence distributive and has the RTPIP. So we may apply Theorem 3.4 in \cite{CzelakowskiDziobiak1990} to conclude that
$\vv Q$ has a finite basis $\Gamma$ of quasiequations. Now $\vv Q$ is a relative congruence distributive dual i-discriminator variety of type $\tau$ if and only if $\vv Q \sse \QQ(\alg K)$  and the conclusion follows.
\end{proof}

To produce $\Gamma$ explicitly from $\tau$ is a different matter. However, some equations are evident; for instance if $f$ is a $k$-ary fundamental operation
of a dual i-discriminator variety $\vv Q$ then
$$
p(x,y,f(\vuc zk)) \app p(x,y,f(p(x,y,z_1),\dots, p(x,y,z_n))
$$
holds in $\vv Q$. As before we say that a fundamental operation $f$ {\bf commutes with} $p$ if for all $\alg A \in \vv Q$ and for all
$a,b,\vuc ck$
$$
p(a,b,f(\vuc ck)) = f(p(a,b,c_1),\dots,p(a,b,c_k).
$$

\begin{lemma}\label{dualendomorphism} Let $\vv Q$ be a  dual i-discriminator variety with i-discriminator term $p(x,y,z)$ and let $\alg A \in \vv Q$. If $C$ is a subclone of $\op{Clo}(\vv Q)$ consisting of operations that commute with $p(x,y,z)$ for any $\alg A \in \vv Q$ and $a,b \in A$ with $a \ne b$ the mapping
$$
z \longmapsto p(a,b,z)
$$
is an endomorphism of $\alg A^C$ whose kernel is $\g_\alg A^\vv Q(a,b)$.
\end{lemma}
\begin{proof} The map is an endomorphism since every operation in $C$ commutes with $p$.  Since $\vv Q$ has the RTPIP, by Theorem \ref{RPIP} we have
$$
\g_\alg A^\vv Q(a,b) = \{(c,d): p(a,b,c) = p(a,b,d)\}.
$$
Therefore the conclusion holds.
\end{proof}

\begin{theorem}\label{i-discriminator complete} Let $\vv Q$ be a dual i-discriminator quasivariety with i-discriminator term $p(x,y,z)$; if  every fundamental operations commutes  with $p(x,y,z)$, then $\vv Q$ is primitive.
\end{theorem}
\begin{proof}  Let $\alg A \in \vv Q$ and let $\th$ be a completely meet irreducible $\vv Q$-congruence of $\alg A$;  then $\alg A/\th$ is $\vv Q$-irreducible and hence $\vv Q$-simple. Let $\Delta = \{\gamma_\alg A^\vv Q(c,d):  \gamma_\alg A^\vv Q(c,d) \le \th\}$ \. Then
\begin{enumerate}
\item  $\bigcup \Delta= \th$ by Theorem \ref{thm: join coprincipal}(2);
\item  $\Delta$ is closed under join by Theorem \ref{thm: join coprincipal}(1);
\item  $\alg A/\th \in \II\SU\PP_u(\alg A)$ by Lemma \ref{dualendomorphism}.
\end{enumerate}
So $\th$ is u-presentable and, by Corollary \ref{cor: cmi and u-presentable}, $\vv Q$ is primitive.
\end{proof}

\begin{corollary}\label{cor:i-discriminator complete}  Let $\vv V$ be a  dual i-discriminator variety with i-discriminator term $p(x,y,z)$; if  every fundamental operations commutes  with $p(x,y,z)$, then $\vv V$ is primitive.
\end{corollary}

Observe that any idempotent operation commutes with $p(x,y,z)$ so every idempotent relative congruence distributive dual i-discriminator quasivariety is primitive. Observe also that an idempotent relative congruence distributive dual i-discriminator variety is a dual discriminator variety in the usual sense \cite{FriedPixley1979}, Corollary \ref{cor:i-discriminator complete} extends the result in \cite{Caicedoetal2021}.

Finally we remark that there are quasivarieties that are  dual i-discriminator, but do not have a relative TD-term
(for instance the variety of distributive lattices). So in this section we covered a different class of examples.

\section{Conclusions and open problems}\label{sec:conclusions}

During our investigation several questions arose (some very general in character, some more specific to the theory) that does not seem to have an immediate answer. We collect them a in series of remarks.

\begin{remark} (Higman's theorem)
Every finitely generated algebra has a (possibly infinite) presentation; if $\alg A$ is finitely generated let $B$ be the set of generators of $\alg A$ let $X =\{x_b: b \in B\}$ and let $\Delta$ be the set of equations in the language of $\alg A$ that are true in $\alg A$ with the substitution $x_b \longmapsto b$.  Then if $\alg A \in \vv K$,  $\alg A \cong \alg F_\vv K(X)/\cg(\Delta)$. At this regard we quote the classical (and beautiful) result of G. Higman \cite{Higman1961}: {\em a finitely generated group is embeddable in a finitely presented group if and only if $\Delta$ is a recursively enumerable set}. The property in this form is probably tailored for groups; however investigating varieties of algebras for which a suitable analogue of Higman's property holds might be worthwhile.\qed
\end{remark}

\begin{remark} (Belkin's lattices) In Example \ref{abgroups} we discussed primitivity in subquasivarieties of abelian groups and modules. This is linked to a more general investigation that we believe worthwhile; Vinogradov \cite{Vinogradov1965}
described all subquasivarieties of abelian groups without investigating the lattice structure. On the other hand it is folklore that the lattice of subvarieties of $\vv M_\alg R$ for any module $\alg R$ is dually isomorphic to the lattice of ideals of $\alg R$. Of course the lattice of subquasivarieties is way more complex in general and Belkin tackled the question in \cite{Belkin1995}; since Belkin's thesis is in Russian and has not been translated, some words of explanation are necessary.  If $I$ is any set let $I^* = I \cup \{\infty\}$; the {\bf Belkin lattice} $\mathfrak B(I)$ of $I$ is the set of functions $f: I^* \longrightarrow \mathbb N^*$ with the following properties
\begin{enumerate}
\ib $f(\infty) \in \{0,\infty\}$
\ib if $f(\infty) =0$, then $f(i) \ne \infty$ for all $i \in I$ and $f(i)=0$ for all (except possibly a finite number) $i \in I$.
\end{enumerate}
If we assume that $\{\infty\}$ is the uppermost element of $\mathbb N^*$, then  $\mathfrak B(I)$ is partially ordered by the natural ordering
$$
f \le g \qquad\text{if and only if}\qquad  f(i) \le g(i)\ \text{for all $i \in I^*$}
$$
and this is a (distributive) lattice ordering. Belkin showed that in case $\alg R$ is a PID, then the lattice of subquasivarieties of $\vv M_\alg R$ is isomorphic with the Belkin lattice of the set of prime ideals of $\alg R$ and his result has been extended to Dedekind domains in \cite{Jedlicka2019}; as abelian groups are $\mathbb Z$-modules this gives a description of the lattice of subquasivarieties of abelian groups. So there are two questions that come naturally:
\begin{enumerate}
\ib can we further generalize Belkin's result? To do so it seems logical to look for some form of decomposition theorem for prime ideals; we do not know if we can do it but for instance the case in which $\alg R$ is a finite direct product of Dedekind domains might be promising;
\ib can we find instances of Belkin's lattices in lattices of subquasivarieties of different structures? This of course depends on the structures but we believe that some cases might be worth investigating.
\end{enumerate}
\end{remark}

\begin{remark} In Section \ref{Ccompleteness} we observed that, if $C$ is a subclone of the term clone of $\vv Q$, then $\vv Q$ can be $C$-structurally complete while  $\vv Q^C$ is not structurally complete.
This happens because $\vv Q$ and $\vv Q^C$ might not have  the same admissible quasiequations, the reason being that the set of available substitutions in $\vv Q$ can be much richer than the one in $\vv Q^C$, making
harder for a rule to be admissible  This problem clearly depends on the specific quasivariety we are dealing with; however we suspect that if we restrict to quasivarieties that are equivalent algebraic semantics of substructural logics, then more can be said.
\end{remark}

\begin{remark} We have shown that having the RTPIP and being relatively filtral are unrelated properties (Examples \ref{demorgan} and \ref{realring}).
An open question is whether any relatively filtral quasivariety with the RTPIP is a dual i-discriminator quasivariety.
\end{remark}

\begin{remark}
Theorem \ref{i-discriminator complete} is a little bit unsatisfactory, in that we would like it to hold also for $C$-structural completeness (and/or $C$-primitivity). The
problem seems to be that the analogue  of Theorem \ref{thm: cmi and u-presentable} does not seem to hold  in the case of relative u-presentability. This deserves further investigation.
\end{remark}

\providecommand{\bysame}{\leavevmode\hbox to3em{\hrulefill}\thinspace}
\providecommand{\MR}{\relax\ifhmode\unskip\space\fi MR }
\providecommand{\MRhref}[2]{%
  \href{http://www.ams.org/mathscinet-getitem?mr=#1}{#2}
}
\providecommand{\href}[2]{#2}

\end{document}